\documentclass[a4paper,twoside,10pt]{article}
\usepackage[a4paper,left=2.5cm,right=2.5cm, top=2.5cm, bottom=2.5cm]{geometry}
\usepackage[latin1]{inputenc}
\usepackage{cuted}

\usepackage{amsthm}
\usepackage{amssymb}
\usepackage{mathrsfs}
\usepackage{mathtools}
\usepackage[shortlabels]{enumitem}
\usepackage{accents}
\usepackage{multirow}
\usepackage{graphicx}
\usepackage{comment}
\usepackage{algorithm2e}
\usepackage{amscd,amsmath,amssymb,mathrsfs,bbm,listings}
\usepackage{cancel}
\usepackage{epstopdf}
\allowdisplaybreaks
\graphicspath{{figs/}}
\usepackage{amsmath}
\usepackage{orcidlink}
\usepackage{footmisc}
\usepackage{amsthm}
\usepackage{amssymb}
\usepackage{stmaryrd}
\SetSymbolFont{stmry}{bold}{U}{stmry}{m}{n}
\usepackage{float}
\usepackage{bigints}
\usepackage{cite}
\usepackage{color}
\usepackage[abs]{overpic}
\usepackage[font=footnotesize,labelfont=bf]{caption}
\usepackage{cases}
\usepackage{tikz}
\usepackage{rotating}
\usepackage{blkarray}
\usetikzlibrary{matrix,calc,arrows,cd}
\usepackage{pdfcomment}
\usepackage{soul,xcolor}
\usepackage{enumitem}
\usepackage{verbatim}
\usepackage{graphicx}
\usepackage{subcaption}
\usepackage{mwe}
\usepackage{tikz}

\newtheorem{theorem}{Theorem}[section]
\newtheorem{lemma}[theorem]{Lemma}
\newtheorem{proposition}[theorem]{Proposition}
\newtheorem{corollary}[theorem]{Corollary}
\newtheorem{remark}[theorem]{Remark}
\newtheorem{definition}[theorem]{Definition}
\newtheorem{assumption}{Assumption}
\newcommand{\eremk}{\hbox{}\hfill\rule{0.8ex}{0.8ex}}

\numberwithin{equation}{section}
\hypersetup{
    colorlinks,
    linktoc=section,
    citecolor=red,
    linkcolor=darkblue,
    urlcolor=magenta,
    pdfborder={0 0 0}
}

\makeatletter
\newcount\my@repeat@count
\newcommand{\pdt}[1]{%
\partial_{
  \begingroup
  \my@repeat@count=\z@
  \@whilenum\my@repeat@count<#1\do{t\advance\my@repeat@count\@ne}%
  \endgroup}
}
\makeatother

\definecolor{darkblue}{rgb}{0.0, 0.0, 0.55}

\newcommand{\CS}{C_S}

\newcommand{\Cp}{C_{\Pi}}

\newcommand{\Cinv}{C_{\mathrm{inv}}}

\newcommand{\Id}{\mathsf{Id}}

\newcommand\In{I_n}

\newcommand\SD{\Sigma_{\mathcal{D}}}

\newcommand\SO{\Sigma_0}
\newcommand\ST{\Sigma_T}
\newcommand\Sn{\Sigma_n}
\newcommand\Qn{Q_n}
\newcommand\tnmo{t_{n-1}}
\newcommand\tn{t_n}

\newcommand\tmmo{t_{m-1}}
\newcommand\tm{t_m}

\newcommand\Th{\mathcal{T}_h}
\newcommand\Tt{\mathcal{T}_{\tau}}
\newcommand{\dV}{\text{d}V}
\newcommand\dx{\mathrm{d}\bx}

\newcommand\ds{\mathrm{d}s}
\newcommand\dt{\mathrm{d}t}

\newcommand\ev{e_v}
\newcommand\eu{e_u}

\newcommand\calR{\mathcal{R}}

\newcommand{\diam}{\mathrm{diam}}

\newcommand{\ut}{u_{\tau}}
\newcommand{\vt}{v_{\tau}}
\newcommand{\wt}{w_{\tau}}
\newcommand{\zt}{z_{\tau}}
\newcommand{\uhtsup}{u_{h, \tau}^{\star}}
\newcommand{\utsup}{u_{\tau}^{\star}}
\newcommand{\uht}{u_{h, \tau}}
\newcommand{\ugh}{u_{h, \tau}^{\calD}}
\newcommand{\vgh}{v_{h, \tau}^{\calD}}
\newcommand{\uoh}{u_{0,h}}
\newcommand{\voh}{v_{0,h}}
\newcommand{\Vhp}{\mathcal{V}_{h}^p}
\newcommand{\Vhpo}{\accentset{\circ}{\mathcal{V}}_{h}^p}

\newcommand{\Vtq}{\mathcal{V}_{\tau}^q}
\newcommand{\VVtq}{\mathbb{V}_{\tau}^q}
\newcommand{\Wtq}{\mathcal{W}_{\tau}^q}
\newcommand{\WWtqmo}{\mathbb{W}_{\tau}^{q-1}}
\newcommand{\WWtqmt}{\mathbb{W}_{\tau}^{q-2}}
\newcommand{\Wtr}{\mathcal{W}_{\tau}^r}
\newcommand{\Vht}{\mathcal{V}_{h, \tau}^{p, q}}
\newcommand{\Vhtqpo}{\accentset{\circ}{\mathcal{V}}_{h, \tau}^{p, q + 1}}
\newcommand{\Vhto}{\accentset{\circ}{\mathcal{V}}_{h, \tau}^{p, q}}
\newcommand{\Wht}{\accentset{\circ}{\mathcal{W}}_{h, \tau}^{p, q - 1}}
\newcommand{\Whts}{{\mathcal{W}}_{h, \tau}^{p, q - 1}}

\newcommand{\Whtt}{\accentset{\circ}{\mathcal{W}}_{h, \tau}^{p, q - 2}}

\newcommand{\zh}{z_h}

\newcommand{\zht}{z_{h, \tau}}
\newcommand{\wht}{w_{h, \tau}}
\newcommand{\vht}{v_{h, \tau}}

\newcommand{\Pih}{\accentset{\circ}{\Pi}_h}
\newcommand{\Rh}{{\mathcal{R}}_h}
\newcommand{\Rho}{\accentset{\circ}{\mathcal{R}}_h}

\newlength{\dhatheight}

\newcommand\IR{\mathbb{R}}

\newcommand\IN{\mathbb{N}}

\newcommand\QT{Q_T}

\newcommand{\gD}{g_{\mathcal{D}}}

\newcommand\bx{\boldsymbol{x}}

\newcommand\bnOmega{{\mathbf{n}}_{\Omega}}

\newcommand\dpt{\partial_t}
\newcommand\dptt{\partial_{tt}}

\newcommand\calD{\mathcal{D}}

\newcommand{\Norm}[2]{\|#1\|_{#2}}

\newcommand{\SemiNorm}[2]{\left|#1\right|_{#2}}

\newcommand{\Pp}[2]{\mathbb{P}^{#1}(#2)}

\newcommand{\Ih}{\mathcal{I}_h}
\newcommand{\Ihg}{\mathcal{I}_h^{\partial}}

\newcommand{\Pt}{\mathcal{P}_{\tau}}
\newcommand{\Piht}{\Pi_{h\tau}}
\newcommand{\epiu}{e_u^{\Pi}}
\newcommand{\epiv}{e_v^{\Pi}}

\title{A variational approach to the analysis of the continuous\\ space--time FEM for the wave equation}
\author{Sergio G\'omez\thanks{Department of Mathematics and Applications, University of Milano-Bicocca, Via Cozzi 55, 20125 Milan, Italy (\href{mailto:sergio.gomezmacias@unimib.it}{sergio.gomezmacias@unimib.it})}\ \thanks{IMATI-CNR ``E. Magenes", Via Ferrata 5, 27100, Pavia, Italy}
\orcidlink{0000-0001-9156-5135}}
\date{}

\begin{document}

\maketitle

\begin{abstract}
\noindent We present a stability and convergence analysis of the space--time continuous finite element method for the Hamiltonian formulation of the wave equation. More precisely, we prove a continuous dependence of the discrete solution on the data in a~$C^0([0, T]; X)$-type energy norm, which does not require any restriction on the meshsize or the time steps.
Such a stability result is then used to derive~\emph{a priori} error estimates with quasi-optimal convergence rates, where a suitable treatment of possible nonhomogeneous Dirichlet boundary conditions is pivotal to avoid loss of accuracy.
Moreover, based on the properties of a postprocessed approximation, we derive a constant-free, reliable~\emph{a posteriori} error estimate in the~$C^0([0, T]; L^2(\Omega))$ norm for the semidiscrete-in-time formulation. Several numerical experiments are presented to validate our theoretical findings.
\end{abstract}

\section{Introduction}
The aim of this work is to present a robust stability and convergence analysis of the space--time continuous finite element method (FEM) originally introduced by French and Peterson in~\cite{French_Peterson:1991} for the wave equation.
Compared to previous works, our analysis is characterized by \emph{i)} the absence of restrictions on the time steps, \emph{ii)} weaker regularity assumptions in the~\emph{a priori} error estimates, \emph{iii)} a less pessimistic dependence of the error growth constants on the final time~$T$, which results from avoiding the use of Gr\"onwall-type estimates, and~\emph{iv)} the treatment of smooth Dirichlet boundary conditions. In addition, we also derive an~\emph{a posteriori} error estimate for the semidiscrete-in-time formulation.

Let the space--time cylinder~$\QT$ be given by~$\Omega \times (0, T)$, where~$\Omega \subset \IR^d$ ($d \in \{1, 2, 3\}$) is a polytopic domain with boundary~$\partial \Omega$, and~$T > 0$ is the final time.
We define the surfaces
$$\SO := \Omega \times \{0\}, \quad \ST := \Omega \times \{T\}, \quad \text{ and } \quad \SD := \partial \Omega \times (0, T).$$ 

Given initial data~$u_0: \Omega \to \IR$ and~$v_0 : \Omega \to \IR$, a Dirichlet datum~$\gD : \SD \to \IR$,
a source term~$f : \QT \to \IR$, and a strictly positive wavespeed~$c : \Omega \to \IR$ independent of the time variable, we consider the following linear acoustic wave problem: find~$u : \QT \to \IR$, such that
\begin{equation}
\label{eq:model-problem}
\begin{cases}
\dptt u - \nabla \cdot( c^2 \nabla u) = f & \text{ in } \QT, \\
u = \gD & \text{ on } \SD, \\
u = u_0 \quad \text{ and }\quad \dpt u = v_0 & \text{ on } \SO.
\end{cases}
\end{equation}
This model can be rewritten in the \emph{so-called} Hamiltonian formulation as follows: find~$u : \QT \to \IR$ and~$v : \QT \to \IR$, such that
\begin{equation}
\label{eq:model-problem-hamiltonian}
\begin{cases}
v = \dpt u & \text{ in } \QT, \\
\dpt v - \nabla \cdot (c^2 \nabla u) = f & \text{ in } \QT, \\
u = \gD & \text{ on } \SD, \\
u = u_0 \quad \text{ and }\quad v = v_0 & \text{ on } \SO.
\end{cases}
\end{equation}

\paragraph{Continuous weak formulation.}
Let the data of the problem satisfy the following assumptions:
\begin{equation*}
\begin{gathered}
f \in L^1(0, T; L^2(\Omega)), \quad \gD \in C^0([0, T]; H^{\frac12}(\partial \Omega)) \cap H^1(0, T; L^2(\partial \Omega)), \\
 u_0 \in H^1(\Omega), 
\quad v_0 \in L^2(\Omega), \quad 
c\in C^0(\overline{\Omega}) \text{ with } 0 < c_{\star} < c(\bx) < c^{\star} \text{ for all } \bx \in \Omega,
\end{gathered}
\end{equation*}
as well as compatibility of the initial and boundary data, i.e., $\gD(\cdot, 0) = u_0(\cdot)$ in~$H^{\frac12}(\partial \Omega)$.
We consider as a weak solution to~\eqref{eq:model-problem} any function~$u$ satisfying:
\begin{itemize}
\item $u \in C^0([0, T]; H^1(\Omega))$, $\dpt u \in C^0([0, T]; L^2(\Omega))$, and~$\dptt u \in L^2(0, T; H^{-1}(\Omega))$;
\item $u_{|_{\SD}}(\cdot, t) = \gD(\cdot, t)$ in~$H^{\frac12}(\partial \Omega)$ for a.e.~$t \in (0, T)$;
\item $u(\cdot, 0) = u_0$ in~$H^1(\Omega)$ and~$\dpt u(\cdot, 0) = v_0$ in~$L^2(\Omega)$;
\item for all~$w \in L^2(0, T; H_0^1(\Omega))$, it holds
\begin{equation*}
\int_0^T \big(\langle\dptt u, w\rangle + (c^2 \nabla u,  \nabla w)_{\Omega} \big)\, \dt = \int_0^T (f, w)_{\Omega}\, \dt,
\end{equation*}
where~$(\cdot, \cdot)_{\Omega}$ and~$\langle \cdot, \cdot \rangle$ denote, respectively, the inner product in~$L^2(\Omega)$ and the duality between~$H^{-1}(\Omega)$ and~$H_0^1(\Omega)$.
\end{itemize}
The proof of existence and uniqueness of a weak solution to~\eqref{eq:model-problem} can be found, e.g., in~\cite[Thm.~2.1 in Part I]{Lasiecka_Triggiani:1994}. Consequently, we consider as the weak solution to~\eqref{eq:model-problem-hamiltonian} the pair~$(u, v)$, with~$u$ as above and~$v = \dpt u \in C^0([0, T]; L^2(\Omega))$.

\paragraph{Previous works.}
Galerkin-type time discretizations are attractive schemes that provide high-order approximations.
Since they are formulated in a variational way, they also allow for an analysis that is closer to the one of the continuous problem.
In addition, their \emph{a priori} error estimates typically require weaker regularity of the continuous solution, as they do not rely on Taylor expansions.

Several Galerkin-type time discretizations have been proposed for the numerical approximation of wave problems. Discontinuous Galerkin (dG) time discretizations propagate the solution from one time slab to the next one by means of upwinding, whereas continuous Galerkin (cG) time discretizations enforce continuity in time of the discrete solution. 

It is well known that the most natural (consistent) dG time discretization for the second-order formulation~\eqref{eq:model-problem} is not always stable and may produce nonphysical oscillations.
This issue was numerically observed by Hulbert and Hughes in~\cite{Hulbert_Hughes:1990}, and was remedied by introducing a least-squares stabilization term. On the other hand, in the presence of a sufficiently strong first-order-in-time damping term, the method regains unconditional stability (see, e.g., \cite{Antonietti_etal:2018}).
The dG class also includes the method by Johnson~\cite{Johnson:1994} for the Hamiltonian formulation~\eqref{eq:model-problem-hamiltonian}, and the method by French~\cite{French:1993} for the second-order formulation~\eqref{eq:model-problem}, where the latter uses exponential-weighted inner products.

The combined dG--cG method by Walkington~\cite{Walkington:2014} for the second-order formulation~\eqref{eq:model-problem} uses continuous approximations in time for the discrete approximation of~$u$, but allows for discontinuities of its first-order time derivative.
Recently, the~$hp$-version of this method was studied in~\cite{Dong-Mascotto-Wang:2024}, and the dG-cG scheme was analyzed for a quasilinear ultrasound wave model in~\cite{Gomez-Nikolic:2024}.
The ingenious techniques used in~\cite{Walkington:2014} to show unconditional stability of the scheme inspired the present work.

It turns out that the most natural continuous-in-space-and-time discretization (i.e., with~$H^1(\QT)$-conforming test and trial spaces) of the second-order formulation~\eqref{eq:model-problem} is only conditionally stable (see~\cite[\S4.2]{Zank:2019}).
Unconditional stability can be obtained by adding a suitable stabilization term~\cite{Zank:2019proc}; see also~\cite{Fraschini_etal:2024,ferrari2024stability} for a high-order B-splines Galerkin version.

We focus on the cG method by French and Peterson~\cite{French_Peterson:1991} for the Hamiltonian formulation~\eqref{eq:model-problem-hamiltonian}, which, unlike all the aforementioned methods, preserves the energy of the scheme when~$\gD = 0$ and~$f = 0$.
Such a property can be crucial in the numerical approximation of relativistic wave models~\cite{Strauss_Vazquez:1978}, 
which were already part of the motivation of~\cite{French_Peterson:1996}. 
Thus, we hope that the present analysis will be useful for designing robust schemes for such models.

Some variations of the method in~\cite{French_Peterson:1991} consider $C^{\ell}$-continuous ($\ell \geq 1$) piecewise polynomial approximations in time; see~\cite{Ferrari_etal:2024} for a B-splines Galerkin version, and~\cite{Anselmann_Bause:2020,Anselmann_etal:2020} for a collocation version. 
In the literature, there are several other methods in the larger class of space--time methods for wave problems, but these are out of the scope of this work.

\paragraph{Main contributions.}
We carry out a robust stability and convergence analysis of the space--time continuous FEM in~\cite{French_Peterson:1991}. 
Compared to the several analyses of such a scheme in previous works, we make the following improvements:
\begin{enumerate}[i)]
    \item We use variational arguments, and do not assume any restrictions on the time steps. This is in contrast to the theoretical CFL condition~$(\tau \le C h)$ in~\cite{French_Peterson:1991}, which results from an unsuitable choice of the time projection in the error analysis. 
    Moreover, techniques relying on the structure of the matrices stemming from the method typically require uniform time steps; see, e.g,, the works by Bales and Lasiecka~\cite{Bales-Lasiecka:1994,Bales-Lasiecka:1995}, as well as the recent analysis in~\cite{Ferrari_etal:2024} of the version using maximal regularity B-splines. 

    \item The constants in both our stability and \emph{a priori} error estimates are independent of the final time~$T$, which is mainly due to the avoidance of Gr\"onwall-type estimates. In fact, the analyses in~\cite{Bales-Lasiecka:1994,Bales-Lasiecka:1995,Zhao-Li:2016,Karakashian_Makridakis:2005} predict an exponential growth in time of the stability and error constants, which is very pessimistic. A milder growth of such constants was obtained in~\cite{French_Peterson:1996}.

    \item We require weaker regularity assumptions of the problem data and the continuous solution. For instance, in the stability estimates, we only require~$f \in L^1(0, T; L^2(\Omega))$, whereas~$f \in L^2(\QT)$ is assumed in~\cite{French_Peterson:1991,Bales-Lasiecka:1994,Bales-Lasiecka:1995,Zhao-Li:2016}, $f \in L^{\infty}(0, T; L^2(\Omega))$ is assumed in~\cite{French_Peterson:1996}, and~$f \in C^0([0, T]; L^2(\Omega))$ is needed when considering the collocation version of the method~\cite[Rem.~3.5]{Bause_etal:2020}.
    As usual, the regularity of~$f$ required in the stability analysis influences the regularity of the continuous solution required in the error estimates.

    \item We address the issue of the strong imposition of smooth nonhomogeneous Dirichlet boundary conditions, which must be done carefully so as to obtain quasi-optimal convergence rates.
    If the Dirichlet datum~$\gD$ is less smooth, one can impose it weakly using a Nitsche's-type method, as discussed in~\cite{Bales-Lasiecka:1994}.

    \item We derive a constant-free, reliable~\emph{a posteriori} error estimate for the cG time discretization of~\eqref{eq:model-problem-hamiltonian}. Similar estimates have been derived in~\cite[\S4]{Johnson:1994} and~\cite[\S3.4]{Makridakis_Nochetto:2006} for the dG time discretization of~\eqref{eq:model-problem-hamiltonian}, and in~\cite[\S3]{Dong-Mascotto-Wang:2024} for the dG-cG time discretization of~\eqref{eq:model-problem}. \emph{A posteriori} error estimates have also been derived in different norms for standard time stepping schemes in~\cite{Adjerid:2002,Bernardi_Suli:2005,Chaumont_Ern:2024,Georgoulis_etal:2016,Grote_etal:2024,Bangerth_etal:2010}.
\end{enumerate}

Denoting by~$\uht$ and~$\vht$ the fully discrete approximations of~$u$ and~$v$, respectively, we prove the following continuous dependence on the data (see Theorem~\ref{thm:continuous-dependence} below):
\begin{alignat*}{3}
\frac12 \big(\Norm{\vht}{C^0([0, T]; L^2(\Omega))}^2 + \Norm{c \nabla \uht}{C^0([0, T]; L^2(\Omega)^d)}^2\big) \lesssim \frac12 \big(\Norm{v_0}{L^2(\Omega)}^2 + \Norm{c\nabla u_0}{L^2(\Omega)^d}^2\big) + \Norm{f}{L^1(0, T; L^2(\Omega))}^2,
\end{alignat*}
where the hidden constant depends only on the degree of approximation in time~$q \geq 1$.

Moreover, in Theorem~\ref{thm:a-priori} below, we prove that, for a sufficiently smooth solution, it holds
\begin{equation*}
\begin{split}
\Norm{v - \vht}{C^0([0, T]; L^2(\Omega))} & \lesssim h^{p + 1} + \tau^{q + 1}, \\
\Norm{c \nabla (u - \uht)}{C^0([0, T]; L^2(\Omega)^d)} & \lesssim h^p + \tau^{q + 1}, \\
\Norm{u - \uht}{C^0([0, T]; L^2(\Omega))} & \lesssim h^{p + 1} + \tau^{q + 1}, 
\end{split}
\end{equation*}
where~$h$ is the meshsize, $\tau$ is the maximum time step, and~$p \geq 1$ and~$q \geq 1$ are the degrees of approximation in space and time, respectively.

Finally, in Theorem~\ref{thm:a-posteriori} below, We derive an~\emph{a posteriori} error estimate of the form
\begin{equation*}
\Norm{u - \ut}{C^0([0, T]; L^2(\Omega))} \le \eta + \mathrm{osc}(f),
\end{equation*}
where~$\eta$ depends on the semidiscrete-in-time solution~$(\ut, \vt)$, and~$\mathrm{osc}(f)$ is an oscillation term. This estimate does not involve unknown constants.

\paragraph{Notation.} 
Given~$m \in \IN$, we denote the~$m$th partial time derivative by~$\dpt^{(m)}$, and by~$\nabla$ and~$\Delta$ the spatial gradient and Laplacian operators.
We use standard notation for~$L^p$, Sobolev, and Bochner spaces. For instance, given~$s \in \IR^+$, $p \in [1, \infty]$, and an open, bounded domain~$\calD \subset \IR^d$ ($d \in \{1, 2, 3\}$) with Lipschitz  boundary~$\partial \cal D$, we denote the corresponding Sobolev space by~$W_p^{s}(\calD)$, its seminorm by~$|\cdot|_{W_p^s(\calD)}$, and its norm by~$\|\cdot\|_{W_p^s(\calD)}$. 
For~$p = 2$, we use the notation~$H^s(\calD) = W_p^s(\calD)$ with seminorm~$|\cdot|_{H^s(\calD)}$ and norm~$\|\cdot\|_{H^s(\calD)}$, and, for~$s = 0$, $H^0(\calD) := L^2(\calD)$ is the space of Lebesgue square integrable functions over~$\calD$ with inner product~$(\cdot, \cdot)_{\calD}$ and norm~$\Norm{\cdot}{L^2(\calD)}$. The closure of~$C_0^{\infty}(\calD)$ in the~$H^1(\calD)$ norm is denoted by~$H_0^1(\calD)$.

Given a Banach space~$X$ and a time interval~$(a, b)$, we denote the corresponding Bochner-Sobolev space by~$W_p^s(a, b; X)$.

We use the following notation for the algebraic tensor product of two vector spaces, say~$V$ and~$W$:
\begin{equation*}
    V \otimes W := \mathrm{span} \big\{vw \, : \, v \in V \text{ and } w \in W \big\}.
\end{equation*}

\paragraph{Structure of the rest of the paper.}
In Section~\ref{sec:description-of-the-method}, we present the space--time continuous FEM, and discuss some alternative formulations.
Section~\ref{sec:continuous-dependence} is devoted to proving a continuous dependence of the discrete solution on the data of the problem, which is based on some nonstandard discrete test functions. 
In Section~\ref{sec:convergence-analysis}, we carry out an~\emph{a priori} error analysis taking into account possible nonhomogeneous Dirichlet boundary conditions, and show quasi-optimal convergence rates in some~$C^0([0, T]; X)$-type norms. 
Based on the properties of the postprocessed approximation presented in Section~\ref{sec:postprocessing}, we derive, in Section~\ref{sec:posteriori}, a constant-free, reliable~\emph{a posteriori} estimate of the error of the semidiscrete-in-time approximation. We present some numerical experiments in~$(2 + 1)$ dimensions in Section~\ref{sec:numerical-exp}, and some concluding remarks in Section~\ref{sec:conclusions}.

\section{Description of the method\label{sec:description-of-the-method}}
In this section, we introduce some notation for the space--time meshes and the discrete spaces used in the description and analysis of the method. After presenting the discrete space--time formulation, we discuss its equivalence to some alternative formulations.

Let~$\{\Th\}_{h > 0}$ be a family of shape-regular conforming simplicial meshes for the spatial domain~$\Omega$, and let~$\Tt$ be a partition of the time interval~$(0, T)$ determined by~$0 =: t_0 < t_1 < \ldots < t_N := T$. For~$n = 1, \ldots, N$, we define the time interval~$\In := (\tnmo, \tn)$, the time step~$\tau_n = \tn - \tnmo$, the surface~$\Sn := \Omega \times \{\tn\}$, and the partial cylinder~$\Qn := \Omega \times \In$. Moreover, we define the meshsize~$h := \max_{K \in \Th} \diam(K)$ and the maximum time step~$\tau := \max_{\In \in \Tt} \tau_n$.

Let~$p, \, q \in \IN$ with~$p \geq 1$ and~$q \geq 1$ denote the degrees of approximation in space and time, respectively. 
We define the following piecewise polynomial spaces:
\begin{subequations}
\begin{alignat}{3}
\Vhp & := \big\{v \in H^1(\Omega) \ : \ v_{|_K} \in \Pp{p}{K} \ \text{ for all } K \in \Th \big\}, \\
\Vhpo & := \Vhp \cap H_0^1(\Omega),\\
\Vtq & := \big\{v \in C^0[0, T] \ : \ v_{|_{\In}} \in \Pp{q}{\In}\ \text{ for } n = 1, \ldots, N \big\},\\
\Wtq & := \big\{v \in L^2(0, T) \ : \ v_{|_{\In}} \in \Pp{q}{\In}\ \text{ for } n = 1, \ldots, N \big\},\\
\label{def:Vht}
\Vht & := \big\{v \in C^0([0, T]; H^1(\Omega)) \ : \ v_{|_{\Qn}} \in \Pp{q}{\In} \otimes \Vhp\ \text{ for } n = 1, \ldots, N \big\}, \\
\Vhto & := \Vht \cap C^0([0, T]; H_0^1(\Omega)), \\
\Whts & := \big\{w \in L^2(0, T; H^1(\Omega)) \ : \ w_{|_{\Qn}} \in \Pp{q-1}{\In} \otimes \Vhp\ \text{ for } n = 1, \ldots, N \big\}, \\
\Wht & := \big\{w \in L^2(0, T; H_0^1(\Omega)) \ : \ w_{|_{\Qn}} \in \Pp{q-1}{\In} \otimes \Vhpo\ \text{ for } n = 1, \ldots, N \big\}.
\end{alignat}
\end{subequations}

Let~$\ugh \in \Vht$ and~$\vgh \in \Vht$ be discrete liftings of~$\gD$ and its time derivative~$\dpt \gD$ (i.e., such that~$\ugh{}_{|_{\partial \Omega}} \approx \gD$ and~$\vgh{}_{|_{\partial \Omega}} \approx \dpt \gD$), respectively. 
Let also~$\uoh \in \Vhp$ and~$\voh \in \Vhp$ be discrete approximations of the initial data~$u_0$ and~$v_0$, respectively. 
Compatibility of~$\ugh$ and~$\uoh$, and of~$\vgh$ and~$\voh$ on~$\partial \Omega \times \{0\}$ is required due to the conformity of the space~$\Vht$ in~\eqref{def:Vht}.
The specific choice of these discrete approximations is discussed in details in Section~\ref{sec:discrete-data} below.

We consider the following space--time formulation: find~$\uht \in \ugh + \Vhto$ and~$\vht \in \vgh + \Vhto$, such that
\begin{subequations}
\label{eq:space-time-formulation}
\begin{alignat}{3}
\label{eq:space-time-formulation-1}
(c^2 \nabla \vht, \nabla \zht)_{\QT} - (c^2 \nabla \dpt \uht, \nabla \zht)_{\QT} & = 0  & & \qquad \forall \zht \in \Wht,\\
\label{eq:space-time-formulation-2}
(\dpt \vht, \wht)_{\QT} + (c^2 \nabla \uht, \nabla \wht)_{\QT} & = (f, \wht)_{\QT}
& & \qquad \forall \wht \in \Wht,
\end{alignat}
\end{subequations}
with discrete initial conditions~$\uht(\cdot, 0) = \uoh$ and~$\vht(\cdot, 0) = \voh$.

Equation~\eqref{eq:space-time-formulation-1} can be seen as a discretization of the relation~$\nabla \cdot (c^2 \nabla v) = \nabla \cdot (c^2 \nabla \dpt u)$ with some boundary conditions.

\begin{remark}[Equivalent formulations]
\label{rem:equivalent-formulations}
The space--time formulation~\eqref{eq:space-time-formulation} describes a Petrov--Galerkin method, as it involves different test and trial spaces. Using the fact that~$\dpt \Vhto = \Wht$, it is possible to rewrite~\eqref{eq:space-time-formulation} as follows: 
\begin{subequations}
\label{eq:space-time-equivalent}
\begin{alignat}{3}
(c^2 \nabla \vht, \nabla \dpt \zht)_{\QT} - (c^2 \nabla \dpt \uht, \nabla \dpt  \zht)_{\QT} & = 0  & & \qquad \forall \zht \in \Vhto,\\
(\dpt \vht, \dpt \wht)_{\QT} + (c^2 \nabla \uht, \nabla \dpt \wht)_{\QT} & = (f, \dpt \wht)_{\QT}
& & \qquad \forall \wht \in \Vhto,
\end{alignat}
\end{subequations}
thus avoiding the use of different test and trial spaces. 
However, we consider that the equivalent form~\eqref{eq:space-time-equivalent} is inconvenient for the analysis, as the nonstandard test functions from the space~$\Wht$ that we use to derive stability estimates in~$C^0([0, T]; X)$-type norms cannot be easily written as the time derivative of some functions from the space~$\Vhto$.

In the case of homogeneous Dirichlet boundary conditions~$(\gD = 0)$, equation~\eqref{eq:space-time-formulation-1} is also equivalent to
\begin{equation}
\label{eq:space-time-equivalence-L2}
(\vht, \zht)_{\QT} - (\dpt \uht, \zht)_{\QT} = 0 \qquad \forall \zht \in \Wht.
\end{equation}
Actually, the inner product used in space is irrelevant, as long as it defines a norm in~$\Vhpo$.
This is due to the fact that~\eqref{eq:space-time-formulation-1} reduces to setting~$\Pi_{q-1}^t \vht = \dpt \uht$, where~$\Pi_{q-1}^t$ denotes the~$L^2(0, T)$-orthogonal projection in~$\mathcal{W}_{\tau}^{q-1}$. On the one hand, such an equivalence is no longer true when~$\gD \neq 0$, and, in general, it is different on the continuous level. 
On the other hand, it turns out that the choice in~\eqref{eq:space-time-formulation-1} is more suitable for the error analysis~(see Remark~\ref{rem:upsilon}) and more robust in practice~(see Section~\ref{sec:nonhomogeneous-dirichlet-experiment}).
\eremk
\end{remark}

\begin{remark}[Discrete lifting functions]
It is necessary to appropriately choose the discrete lifting functions~$\ugh$ and~$\vgh$ in order to obtain quasi-optimal a priori error estimates. 
In particular, they must involve the projection in time~$\Pt$ (see Section~\ref{sec:Pt}) used in the convergence analysis.
This is not only a theoretical issue, as a ``naive" treatment of nonhomogeneous Dirichlet boundary conditions may actually lead to suboptimal convergence rates; see, e.g., the numerical experiment in Section~\ref{sec:nonhomogeneous-dirichlet-experiment} below.
Such an important observation was one of the main focus in the analysis of the combined dG--cG method in~\cite{Walkington:2014}.
\eremk
\end{remark}
\section{Stability analysis\label{sec:continuous-dependence}}
In this section, we study the stability of the space--time FEM~\eqref{eq:space-time-formulation}. More precisely, in Theorem~\ref{thm:continuous-dependence} below, we show that the discrete solution to~\eqref{eq:space-time-formulation} satisfies a continuous dependence on the data in some~$C^0([0, T]; X)$-type norms with no restrictions on the meshsize~$h$ or the time step~$\tau$. 

In the remaining of this section, we assume that~$\gD = 0$, as the stability of the scheme for nonhomogeneous Dirichlet boundary conditions follows by a translation argument.

\subsection{Some auxiliary results for the stability analysis}
We first recall the stability properties of some standard orthogonal projections, and introduce some auxiliary weight functions that we use to prove the continuous dependence of the discrete solution on the data of the problem.
\subsubsection{Orthogonal projections and inverse estimates} 
We denote by~$\Pih : L^2(\Omega) \to \Vhpo$ the~$L^2(\Omega)$-orthogonal projection in~$\Vhpo$. Moreover, we denote by~$\Rho : H^1(\Omega) \to \Vhpo$ the weighted Ritz projection operator, which is defined for any~$\varphi \in H^1(\Omega)$ as the solution to the following variational problem:
\begin{equation}
\label{def:Rh}
(c^2 \nabla \Rho \varphi, \nabla \zh)_{\Omega} = (c^2 \nabla \varphi, \nabla \zh)_{\Omega} \qquad \forall \zh \in \Vhpo.
\end{equation}
Henceforth, the symbol~$\circ$ is used to recall that these projections map into the space of polynomials with zero trace~$\partial \Omega$. 

Moreover, given~$r \in \IN$, we denote by~$\Pi_r^t : L^2(0, T) \to \Wtr$ the~$L^2(0, T)$-orthogonal projection in~$\Wtr$.
In next lemmas, we present some stability results for the orthogonal projection~$\Pi_r^t$, as well as some polynomial inverse estimates.
\begin{lemma}[Stability of~$\Pi_r^t$, {see~\cite[Thm.~18.16(ii)]{Ern_Guermond-book-I}}]
\label{lemma:stab-pi-time}
Let~$r \in \IN$ and~$s \in [1, \infty]$. There exists a positive constant~$\CS$ independent of~$\tau$ such that
\begin{equation*}
\Norm{\Pi_r^t v}{L^s(0, T)} \le \CS \Norm{v}{L^s(0, T)} \qquad \forall v \in L^s(0, T) \cap L^2(0, T).
\end{equation*}
\end{lemma}

\begin{lemma}[Polynomial inverse estimates, {see~\cite[Thm.~4.5.11]{Brenner-Scott:book}}]
\label{lemma:inverse-estimates}
Let~$r \in \IN$ and~$s \in [1, \infty]$. There exists a positive constant~$\Cinv$ independent of~$\tau$, such that, for~$n = 1, \ldots, N$, it holds
\begin{subequations}
\begin{alignat}{3}
\label{eq:inverse-polynomial-time}
\Norm{\wt'}{L^s(\In)} & \le \Cinv \tau_n^{-1} \Norm{\wt}{L^s(\In)} & \qquad \forall \wt \in \Pp{r}{\Tt}, \\
\label{eq:L-infty-L2}
\Norm{\wt}{L^{\infty}(\In)} & \le (1 + \Cinv) \tau_n^{-\frac12} \Norm{\wt}{L^2(\Omega)} & \qquad \forall \wt \in \Pp{r}{\Tt}.
\end{alignat}
\end{subequations}
\end{lemma}
\begin{proof}
The inverse estimate~\eqref{eq:L-infty-L2} follows by combining the inequality~$\Norm{\wt}{L^{\infty}(\In)} \le \tau_n^{-1/2} \Norm{\wt}{L^2(\In)} + \tau_n^{1/2} \Norm{\wt'}{L^2(\In)}$ (see, e.g., \cite[Eq.~(1.9)]{Ern_Guermond-book-I}) with the inverse estimate~\eqref{eq:inverse-polynomial-time}.
\end{proof}

In addition, we denote by~$\Id$ the identity operator. In what follows, the projections~$\Pih$ and~$\Rho$ are to be understood as applied pointwise in time, whereas the projection~$\Pi_r^t$ is to be understood as applied pointwise in space.

\subsubsection{Auxiliary weight functions}
We now introduce some auxiliary weight functions similar to those used in~\cite{Walkington:2014,Gomez-Nikolic:2024,Dong-Mascotto-Wang:2024} to derive continuous dependence on the data of the discrete solution for the dG--cG scheme in~$L^{\infty}(0, T; X)$-type norms. 

For~$n = 1, \ldots, N$, we define the following linear polynomial:
\begin{equation*}
\varphi_n(t) := 1 - \lambda_n (t - \tnmo), \quad \text{ with } \lambda_n := \frac{1}{2\tau_n}.
\end{equation*}
These functions satisfy the uniform bounds
\begin{subequations}
\begin{alignat}{3}
\label{eq:uniform-bound-varphi}
0 < \frac12 \le \varphi_n(t) \le 1  & \quad \forall t \in [\tnmo, \tn], \\
\varphi_n'(t) = -\lambda_n & \quad \forall t \in [\tnmo, \tn].
\end{alignat}
\end{subequations}

\subsection{Continuous dependence on the data}
Standard energy arguments lead to bounds of the discrete solution in the energy norm only at the discrete times~$\{\tn\}_{n = 1}^N$, which is not enough to control the energy at all times. We call this
a~``\emph{weak partial bound}" of the discrete solution (see Proposition~\ref{prop:weak-continuous-wave}).
Such a result can be used in combination with nonstandard discrete test functions to prove continuous dependence on the data in  the energy norm at all times (see Theorem~\ref{thm:continuous-dependence}).

In what follows, we use the notation~$a \lesssim b$ to indicate the existence of a positive constant~$C$ independent of the meshsize~$h$ and the maximum time step~$\tau$ such that~$a \le C b$. 
Similarly, we use~$a \simeq b$ meaning that~$a \lesssim b$ and~$b \lesssim a$.

For convenience, we study the stability properties of the solution to the following perturbed space--time formulation: find~$(\uht, \vht) \in \Vhto \times \Vhto$, such that
\begin{subequations}
\label{eq:space-time-formulation-perturbed}
\begin{alignat}{3}
\label{eq:space-time-formulation-perturbed-1}
(c^2 \nabla \vht, \nabla \zht)_{\QT} - (c^2 \nabla \dpt \uht, \nabla \zht)_{\QT} & = (c^2 \nabla \Upsilon, \nabla \zht)_{\QT} & & \qquad \forall \zht \in \Wht,\\
\label{eq:space-time-formulation-perturbed-2}
(\dpt \vht, \wht)_{\QT} + (c^2 \nabla \uht, \nabla \wht)_{\QT} & = (f, \wht)_{\QT}
& & \qquad \forall \wht \in \Wht,
\end{alignat}
\end{subequations}
where~$\Upsilon \in L^2(0, T; H^1(\Omega))$
will be used to represent some terms related to the projection error in the \emph{a priori} error analysis in Section~\ref{sec:convergence-analysis} below. 

\begin{remark}[Choice of~$\uoh$ and~$\voh$]
It is convenient to set the discrete initial conditions of the space--time formulation~\eqref{eq:space-time-formulation-perturbed} as follows:
\begin{equation}
\label{eq:discrete-initial}
\uoh := \Rho u_0 \quad \text{and} \quad \voh := \Pih v_0.
\end{equation}
Taking~$\voh = \Pih v_0$ allows for~$v_0 \in L^2(\Omega)$ in the stability estimate in Theorem~\ref{thm:continuous-dependence} below. 
The assumption~$v_0 \in H^1(\Omega)$ is necessary later in the convergence analysis in Section~\ref{sec:convergence-analysis}. 
For the~$C^{\ell}$-conforming ($\ell \in \IN$) collocation version of the space--time method~\eqref{eq:space-time-formulation}, additional nonstandard discrete initial and transmission conditions have to be computed, which demand higher regularity of the continuous solution (see~\cite[Problem 3.3]{Anselmann_etal:2020} and~\cite[Rem.~2]{Anselmann_Bause:2020}).

For nonhomogeneous Dirichlet boundary conditions, the choice in~\eqref{eq:discrete-initial} has to be modified so as to ensure compatibility with the discrete liftings~$\ugh$ and~$\vgh$ (see Section~\ref{sec:discrete-data} for more details).
\eremk
\end{remark}

\begin{proposition}[Weak partial bound on the discrete solution]
\label{prop:weak-continuous-wave}
Let~$(\uht, \vht) \in \Vhto \times \Vhto$ be a solution to the discrete space--time formulation~\eqref{eq:space-time-formulation-perturbed} with discrete initial conditions given by~\eqref{eq:discrete-initial}. For~$n = 1, \ldots, N$, the following bound holds:
\begin{equation}
\label{eq:weak-partial-bound}
\begin{split}
\frac12 \big( \Norm{\vht}{L^2(\Sn)}^2 + \Norm{c \nabla \uht}{L^2(\Sn)^d}^2\big) & \le \frac12 \big(\Norm{v_0}{L^2(\Omega)}^2 + \Norm{c \nabla u_0}{L^2(\Omega)^d}^2 \big) \\
& \quad + \CS \Norm{f}{L^1(0, \tn; L^2(\Omega))} \Norm{\vht}{L^{\infty}(0, \tn; L^2(\Omega))} \\
& \quad + \CS \Norm{c \nabla \Upsilon}{L^1(0, \tn; L^2(\Omega)^d)} \Norm{c \nabla \uht}{L^{\infty}(0, \tn; L^2(\Omega)^d)},
\end{split}
\end{equation}
where~$\CS$ is the constant in Lemma~\ref{lemma:stab-pi-time}, which depends only on~$q$.
\end{proposition}
\begin{proof}
Without loss of generality, we prove the result for the case~$n = N$.

We choose~$\wht = \Pi_{q-1}^t \vht \in \Wht$ in~\eqref{eq:space-time-formulation-perturbed-2} and obtain the following identity:
\begin{alignat}{3}
\label{eq:aux-identity-weak-bound}
(\dpt \vht, \Pi_{q-1}^t \vht)_{\QT} + (c^2 \nabla \uht, \nabla \Pi_{q-1}^t \vht)_{\QT} & = (f, \Pi_{q-1}^t \vht)_{\QT}.
\end{alignat}

We first consider the terms on the left-hand side of~\eqref{eq:aux-identity-weak-bound}. Using the properties of the~$L^2(0, T)$-orthogonal projection~$\Pi_{q-1}^t$, the continuity in time of~$\vht$, and the choice of the discrete initial conditions in~\eqref{eq:discrete-initial}, we have
\begin{equation}
\label{eq:identity-dpt}
(\dpt \vht, \Pi_{q-1}^t \vht)_{\QT} = (\dpt \vht, \vht)_{\QT} = \frac12 \big(\Norm{\vht}{L^2(\ST)}^2 - \Norm{\Pih v_0}{L^2(\Omega)}^2 \big).
\end{equation}

As for the second term on the left-hand side of~\eqref{eq:aux-identity-weak-bound}, we use~\eqref{eq:space-time-formulation-perturbed-1}, the commutativity of the spatial gradient~$\nabla$ and the first-order time derivative~$\dpt$, the continuity in time of~$\uht$, and the choice of the discrete initial conditions in~\eqref{eq:discrete-initial} to obtain
\begin{alignat}{3}
\nonumber
(c^2 \nabla \uht, \nabla \Pi_{q-1}^t \vht)_{\QT} & = (\nabla \Pi_{q-1}^t \uht, c^2 \nabla \vht)_{\QT} \\
\nonumber
& = (\nabla \Pi_{q-1}^t \uht, c^2 \nabla \dpt \uht)_{\QT} + (\nabla \Pi_{q-1}^t \uht, c^2 \nabla \Upsilon)_{\QT} \\
\nonumber
& = (\nabla \uht, c^2 \nabla \dpt \uht)_{\QT} + (c \nabla \uht, c \Pi_{q-1}^t \nabla \Upsilon)_{\QT} \\
\label{eq:identity-grad-grad}
& = \frac12 \big(\Norm{c \nabla \uht}{L^2(\ST)^d}^2 - \Norm{c \nabla \Rho u_0}{L^2(\Omega)^d}^2 \big) + (\Pi_{q-1}^t (c \nabla \Upsilon), c \nabla \uht)_{\QT} .
\end{alignat}

Inserting identities~\eqref{eq:identity-dpt} and~\eqref{eq:identity-grad-grad} in~\eqref{eq:aux-identity-weak-bound}, and using the stability properties of~$\Pih$ and~$\Rho$, the H\"older inequality, and Lemma~\ref{lemma:stab-pi-time}, we get
\begin{alignat*}{3}
\frac12 \big( \Norm{\vht}{L^2(\ST)}^2 & + \Norm{c\nabla \uht}{L^2(\ST)^d}^2\big)\\
& = \frac12 \big(\Norm{\Pih v_0}{L^2(\Omega)}^2 + \Norm{c \nabla \Rho u_0}{L^2(\Omega)^d}^2 \big)  + (f, \Pi_{q-1}^t \vht)_{\QT} - (\Pi_{q-1}^t (c\nabla \Upsilon), c \nabla \uht)_{\QT} \\
& \le \frac12 \big(\Norm{v_0}{L^2(\Omega)}^2 + \Norm{c \nabla u_0}{L^2(\Omega)^d}^2 \big) + \Norm{f}{L^1(0, T; L^2(\Omega))} \Norm{\Pi_{q-1}^t \vht}{L^{\infty}(0, T; L^2(\Omega))} \\
& \quad + \Norm{\Pi_{q-1}^t (c\nabla \Upsilon) }{L^1(0, T; L^2(\Omega)^d)} \Norm{c \nabla \uht}{L^{\infty}(0, T; L^2(\Omega)^d)} \\
& \le \frac12 \big(\Norm{v_0}{L^2(\Omega)}^2 + \Norm{c \nabla u_0}{L^2(\Omega)^d}^2 \big) + \CS \Norm{f}{L^1(0, T: L^2(\Omega))} \Norm{\vht}{L^{\infty}(0, T; L^2(\Omega))} \\
& \quad + \CS \Norm{c \nabla \Upsilon}{L^1(0, T; L^2(\Omega)^d)} \Norm{c \nabla \uht}{L^{\infty}(0, T; L^2(\Omega)^d)},
\end{alignat*}
where~$\CS$ is the constant in Lemma~\ref{lemma:stab-pi-time}. This completes the proof of~\eqref{eq:weak-partial-bound}.
\end{proof}

We are now in a position to prove the main result in this section.

\begin{theorem}[Continuous dependence on the data]
\label{thm:continuous-dependence}
Let~$(\uht, \vht) \in \Vhto \times \Vhto$ be a solution to the discrete space--time formulation~\eqref{eq:space-time-formulation-perturbed} with discrete initial conditions given by~\eqref{eq:discrete-initial}. Then, the following bound holds:
\begin{equation}
\label{eq:continuous-dependence-perturbed}
\begin{split}
& \frac12 \Big(\Norm{\vht}{C^0([0, T]; L^2(\Omega))}^2 + \Norm{c \nabla \uht}{C^0([0, T]; L^2(\Omega)^d)}^2 \Big) \\
& \qquad \qquad \lesssim  \frac12 \big(\Norm{v_0}{L^2(\Omega)}^2 + \Norm{c \nabla u_0}{L^2(\Omega)^d}^2 \big) + \Norm{f}{L^1(0, T; L^2(\Omega))}^2 + \Norm{c\nabla \Upsilon}{L^1(0, T; L^2(\Omega)^d)}^2,
\end{split}
\end{equation}
where the hidden constant depends only on~$q$.
\end{theorem}
\begin{proof}

Let~$n \in \{1, \ldots, N\}$. We define the following test function:
\begin{equation*}
\wht^{\star, n} {}_{|_{Q_m}} := \begin{cases}
\Pi_{q-1}^t (\varphi_n \Pi_{q-1}^t \vht) & \text{ if } m = n, \\
0 & \text{ otherwise}.
\end{cases} 
\end{equation*}
Taking~$\wht^{\star, n} \in \Wht$ as the test function in~\eqref{eq:space-time-formulation-perturbed-2}, we get
\begin{alignat}{3}
\label{eq:aux-identity-linear-wave}
(\dpt \vht, \Pi_{q-1}^t (\varphi_n \Pi_{q-1}^t \vht) )_{\Qn} + (c^2 \nabla \uht, \nabla \Pi_{q-1}^t (\varphi_n \Pi_{q-1}^t \vht))_{\Qn} = (f, \Pi_{q-1}^t (\varphi_n \Pi_{q-1}^t \vht))_{\Qn}.
\end{alignat}

\paragraph{Bound on~$(\dpt \vht, \wht^{\star, n})_{\Qn}$.} Using the orthogonality properties of~$\Pi_{q-1}^t$, the first term on the left-hand side of~\eqref{eq:aux-identity-linear-wave} can be split as follows:
\begin{alignat*}{3}
(\dpt \vht, \Pi_{q-1}^t (\varphi_n \Pi_{q-1}^t \vht))_{\Qn} 
&  = (\dpt \vht, \varphi_n \Pi_{q-1}^t \vht)_{\Qn}  \\
& = (\varphi_n \dpt \vht,  \vht)_{\Qn} - (\varphi_n \dpt \vht, (\Id - \Pi_{q-1}^t ) \vht)_{\Qn} \\
& =: J_1 + J_2.
\end{alignat*}
Using the identity
$$\varphi_n w \dpt w = \frac12 \dpt(\varphi_n w^2) - \frac12 \varphi' w^2,$$
and the fact that~$\varphi_n(\tnmo) = 1$, $\varphi_n(\tn) = \frac12$, and~$\varphi' = -\lambda_n$, we have
\begin{alignat}{3}
\nonumber
J_1 = (\varphi_n \dpt \vht, \vht)_{\Qn} & = \frac12 \int_{\Qn} \dpt (\varphi_n \vht)^2 \dV  - \frac12 \int_{\Qn} \varphi_n' \vht^2 \dV \\
\label{eq:J1}
& = \frac{1}{4} \Norm{\vht}{L^2(\Sn)}^2 - \frac12 \Norm{\vht}{L^2(\Sigma_{n - 1})}^2 + \frac{\lambda_n}{2} \Norm{\vht}{L^2(\Qn)}^2.
\end{alignat}

Let now~$\{L_r\}_{r = 0}^q$ denote the Legendre polynomials defined in the time interval~$\In$. Since~$\vht \in \Vhto$, there exists a function~$\alpha_q \in \Vhp$ such that
\begin{subequations}
\begin{alignat}{3}
\label{eq:identity-(I-Pi)-vht}
(\Id - \Pi_{q-1}^t) \vht(\bx, t) & = \alpha_q(\bx) L_q(t),\\
\label{eq:expansion-dpt-vht}
\dpt \vht(\bx, t) & = \dpt \Pi_{q-1}^t \vht(\bx, t) + \alpha_q(\bx) L_q'(t).
\end{alignat}
\end{subequations}
Moreover, since~$\dpt \Pi_{q-1}^t \vht \in \Whtt$, the definition~$\varphi_n(t) = 1 - \lambda_n (t - \tnmo)$, identities~\eqref{eq:identity-(I-Pi)-vht} and~\eqref{eq:expansion-dpt-vht}, and the orthogonality properties of~$\Pi_{q-1}^t$ and the Legendre polynomials~$\{L_r\}_{r = 0}^q$ lead to
\begin{alignat}{3}
\nonumber
J_2 & = -(\varphi_n \dpt \vht,\, (\Id - \Pi_{q-1}^t) \vht)_{\Qn}  = -(\varphi_n \dpt \Pi_{q-1}^t \vht + \alpha_q \varphi_n L_q',\, (\Id - \Pi_{q-1}^t) \vht)_{\Qn} \\
\nonumber
& = \lambda_n (\alpha_q (t - \tnmo) L_q',\, \alpha_q L_q)_{\Qn} 
\\
\nonumber
& = \lambda_n \Norm{\alpha_q}{L^2(\Omega)}^2 \int_{\tnmo}^{\tn} (t - \tnmo) L_q(t) L_q'(t) \dt  \\
\label{eq:J2}
& = \tau_n \lambda_n \Big(\frac{q}{2q + 1}\Big) \Norm{\alpha_q}{L^2(\Omega)}^2 = \frac12 \Big(\frac{q}{2q + 1}\Big) \Norm{\alpha_q}{L^2(\Omega)}^2 \geq 0,
\end{alignat}
where, in the last line, we have used the identities~$\tau_n \lambda_n = 1/2$ and 
\begin{equation*}
\int_{\tnmo}^{\tn} (t - \tnmo) L_q(t) L_q'(t) \dt = \frac{\tau_n q}{2q + 1}.
\end{equation*}

\paragraph{Bound on~$(c^2 \nabla \uht, \nabla \wht^{\star, n})_{\Qn}$.} Using~\eqref{eq:space-time-formulation-perturbed-1} in the definition of the perturbed problem and the orthogonality properties of~$\Pi_{q-1}^t$, the second term on the left-hand side of~\eqref{eq:aux-identity-linear-wave} can be split as follows:
\begin{alignat*}{3}
(c^2 \nabla \uht, \nabla \Pi_{q-1}^t(\varphi_n \Pi_{q-1}^t \vht))_{\Qn} 
& = (c^2 \varphi_n \nabla \Pi_{q-1}^t \uht, \nabla \Pi_{q-1}^t \vht))_{\Qn} \\
& = (\Pi_{q-1}^t(\varphi_n \nabla \Pi_{q-1}^t \uht), c^2 \nabla \vht)_{\Qn} \\
& = (\varphi_n \nabla \Pi_{q-1}^t \uht, c^2 \nabla \dpt \uht)_{\Qn} + (\Pi_{q-1}^t(\varphi_n \nabla \Pi_{q-1}^t \uht), c^2 \nabla \Upsilon)_{\Qn} \\
& = (c^2 \nabla \uht, \nabla \Pi_{q-1}^t (\varphi_n \dpt \uht) )_{\Qn} + (c^2 \nabla \uht, \nabla \Pi_{q-1}^t (\varphi_n \Pi_{q-1}^t \Upsilon) )_{\Qn}  \\
& = (c^2 \varphi_n \nabla \uht, \nabla \dpt \uht)_{\Qn} - (c^2 \nabla \uht, (\Id - \Pi_{q-1}^t) (\varphi_n \nabla \dpt \uht))_{\Qn}  \\
& \quad + (c \nabla \uht, \Pi_{q-1}^t (\varphi_n \Pi_{q-1}^t (c \nabla \Upsilon)) )_{\Qn} \\
& =: K_1 + K_2 + K_3.
\end{alignat*}

Similarly as in identity~\eqref{eq:J1} for~$J_1$, we obtain
\begin{alignat*}{3}
K_1 = (c^2 \varphi_n \nabla \uht, \nabla \dpt \uht)_{\Qn} = \frac{1}{4} \Norm{c \nabla \uht}{L^2(\Sn)^d}^2 - \frac12 \Norm{c \nabla \uht}{L^2(\Sigma_{n - 1})^d}^2 + \frac{\lambda_n}{2} \Norm{c \nabla \uht}{L^2(\Qn)^d}^2.
\end{alignat*}
Moreover, the same reasoning in~\eqref{eq:J2} for~$J_2$ can be used to show that~$K_2 \geq 0$.

Using the H\"older inequality, the stability bound in Lemma~\ref{lemma:stab-pi-time} for~$\Pi_{q-1}^t$, and the uniform bound in~\eqref{eq:uniform-bound-varphi} for~$\varphi_n$, we get
\begin{alignat*}{3}
K_3 = (c \nabla \uht, \Pi_{q-1}^t (\varphi_n \Pi_{q-1}^t (c \nabla \Upsilon)) )_{\Qn}
& \le \Norm{\Pi_{q-1}^t (\varphi_n \Pi_{q-1}^t (c \nabla \Upsilon))}{L^1(\In; L^2(\Omega)^d)} \Norm{c \nabla \uht}{L^{\infty}(\In; L^2(\Omega)^d)} \\
& \le \CS^2 \Norm{c \nabla \Upsilon}{L^1(\In; L^2(\Omega)^d)} \Norm{c \nabla \uht}{L^{\infty}(\In; L^2(\Omega)^d)}.
\end{alignat*}

\paragraph{Bound on~$(f, \wht^{\star, n})_{\Qn}$.}  We now focus on the term in~\eqref{eq:aux-identity-weak-bound} involving the source term~$f$. The H\"older inequality, the stability bound in Lemma~\ref{lemma:stab-pi-time} for~$\Pi_{q-1}^t$, and the uniform bound in~\eqref{eq:uniform-bound-varphi} for~$\varphi_n$ lead to
\begin{alignat*}{3}
(f, \Pi_{q-1}^t(\varphi_n \Pi_{q-1}^t \vht))_{\Qn} & \le \Norm{f}{L^1(\In; L^2(\Omega))} \Norm{\Pi_{q-1}^t (\varphi_n \Pi_{q-1}^t \vht)}{L^{\infty}(\In; L^2(\Omega))} \\
& \le \CS^2 \Norm{f}{L^1(\In; L^2(\Omega))} \Norm{\vht}{L^{\infty}(\In; L^2(\Omega))}.
\end{alignat*}

\paragraph{Conclusion.} Using the inverse estimate~\eqref{eq:L-infty-L2}, we have
\begin{alignat*}{3}
\frac{1}{4(1 + \Cinv)^2} \Norm{\vht}{L^{\infty}(\In; L^2(\Omega))}^2 & \le \frac{\lambda_n}{2}\Norm{\vht}{L^2(\Qn)}^2, \\
\frac{1}{4(1 + \Cinv)^2} \Norm{c \nabla \uht}{L^{\infty}(\In; L^2(\Omega)^d)}^2 & \le \frac{\lambda_n}{2} \Norm{c \nabla \uht}{L^2(\Qn)^d}^2.
\end{alignat*}
Therefore, combining the above estimates, and using Proposition~\ref{prop:weak-continuous-wave} to bound the energy terms at~$\tnmo$, we obtain
\begin{alignat}{3}
\nonumber
\frac{1}{4} & \big(\Norm{\vht}{L^2(\Sn)}^2 + \Norm{c \nabla \uht}{L^2(\Sn)^d}^2 \big) + \frac{1}{4(1 + \Cinv)^2} \big(\Norm{\vht}{L^{\infty}(\In; L^2(\Omega))}^2 + \Norm{c \nabla \uht}{L^{\infty}(\In; L^2(\Omega)^d)}^2 \big) \\
\nonumber
& \le \frac12 \big(\Norm{\vht}{L^2(\Sigma_{n - 1})}^2 + \Norm{c \nabla \uht}{L^2(\Sigma_{n - 1})^d}^2 \big) +  \CS^2 \Norm{f}{L^1(\In; L^2(\Omega))} \Norm{\vht}{L^{\infty}(\In; L^2(\Omega))} \\
\nonumber 
& \quad + \CS^2 \Norm{c \nabla \Upsilon}{L^1(\In; L^2(\Omega)^d)} \Norm{c \nabla \uht}{L^{\infty}(\In; L^2(\Omega)^d)} \\
\nonumber
& \le \frac12\big(\Norm{v_0}{L^2(\Omega)}^2 + \Norm{c \nabla u_0}{L^2(\Omega)^d}^2 \big) \\
\nonumber
& \quad + \CS \Norm{f}{L^1(0, \tnmo; L^2(\Omega))} \Norm{\vht}{L^{\infty}(0, \tnmo; L^2(\Omega))} 
\\
\nonumber
& \quad + \CS^2 \Norm{f}{L^1(\In; L^2(\Omega))} \Norm{\vht}{L^{\infty}(\In; L^2(\Omega))}\\
\nonumber
& \quad + \CS \Norm{c\nabla \Upsilon}{L^1(0, \tnmo; L^2(\Omega)^d))} \Norm{c \nabla \uht }{L^{\infty}(0, \tnmo; L^2(\Omega)^d)} \\
\nonumber 
& \quad + \CS^2 \Norm{c \nabla \Upsilon}{L^1(\In; L^2(\Omega)^d)} \Norm{c \nabla \uht }{L^{\infty}(\In; L^2(\Omega)^d)}\\ 
\nonumber
& \lesssim \frac12 \big(\Norm{v_0}{L^2(\Omega)}^2 + \Norm{c \nabla u_0}{L^2(\Omega)^d}^2  \big) + \Norm{f}{L^1(0, \tn; L^2(\Omega))} \Norm{\vht}{L^{\infty}(0, \tn; L^2(\Omega))} \\
\label{eq:bound-energy-n}
& \quad + \Norm{c \nabla \Upsilon}{L^1(0, \tn; L^2(\Omega)^d)} \Norm{c \nabla \uht}{L^{\infty}(0, T; L^2(\Omega)^d)},
\end{alignat}
where the hidden constant depends only on~$q$.

The desired result then follows by considering the value of~$n$ where the left-hand side of~\eqref{eq:bound-energy-n} takes its maximum value, and using the Young inequality. 
\end{proof}

\begin{remark}[Stability of the space--time FEM~\eqref{eq:space-time-formulation}]
A stability bound on the discrete solution~$(\uht, \vht)$ to the space--time FEM~\eqref{eq:space-time-formulation} can be retrieved by setting~$\Upsilon = 0$ in Theorem~\ref{thm:continuous-dependence}. This implies that the method is stable with no restrictions on the time steps~$\{\tau_n\}_{n = 1}^N$ or additional assumptions on the family of shape-regular meshes~$\{\Th\}_{h > 0}$. In order to obtain an actual continuous dependence on the data result, the particular choice of the discrete initial conditions in~\eqref{eq:discrete-initial} is highly relevant. 

The hidden constant in the statement of Theorem~\ref{thm:continuous-dependence} is independent of the final time~$T$, which significantly improves the exponential growth of the stability constant predicted in the analyses in~\cite{Bales-Lasiecka:1994,Bales-Lasiecka:1995,Zhao-Li:2016,Karakashian_Makridakis:2005}, which results from the use of Gr\"onwall-type estimates. Finally, we require~$f \in L^1(0, T; L^2(\Omega))$ instead of the standard regularity assumption~$f \in L^2(\QT)$.
\eremk
\end{remark}

\begin{remark}[Changing meshes]
The continuity in time enforced in the discrete space~$\Vht$ prevents the use of spatial meshes that change from one time slab to the next one. 
One could add a transfer condition~$\uht(\cdot, \tnmo^+) = \Pi_h^{(n)} \uht(\cdot, \tnmo^-)$ and~$\vht(\cdot, \tnmo^+) = \Pi_h^{(n)} \vht(\cdot, \tnmo^-)$ for some projection~$\Pi_h^{(n)}$ onto the space~$\Pp{q}{\In} \otimes \mathcal{V}_h^{(n), p}$, as done in~\cite{Karakashian_Makridakis:2005} for semilinear wave equations. However, this significantly complicates the stability and convergence analysis.
\eremk
\end{remark}

\begin{remark}[Perturbation term~$\Upsilon$]
\label{rem:upsilon}
From the theoretical point of view, the main difference between equation~\eqref{eq:space-time-formulation-perturbed-1} and
\begin{equation*}
(\vht, \zht) - (\dpt \uht, \zht) = (\Upsilon, \zht) \qquad \forall \zht \in \Wht
\end{equation*}
is that the latter leads to an undesirable term~$\Norm{c \nabla \Pih \Upsilon}{L^1(0, T; L^2(\Omega)^d)}$ in~\eqref{eq:continuous-dependence-perturbed}, which results in stronger conditions on~$\Th$ to ensure stability of~$\Pih$ in the~$H^1(\Omega)$ seminorm (see~\cite{Tantardini_Veeser:2016}), and is difficult to handle in the a priori error analysis.
\eremk
\end{remark}

\begin{remark}[Insights into Neumann boundary conditions]
Our analysis can be adapted to accommodate Neumann boundary conditions $\big(c^2 \nabla u \cdot \bnOmega = g$ on~$\partial \Omega \times (0, T) \big)$. 
Given a Neumann datum~$g \in W^{1,1}(0, T; H^{-1/2}(\partial \Omega)) \hookrightarrow C^0([0, T]; H^{-1/2}(\partial \Omega))$, the resulting space--time formulation reads: find~$\uht \in \Vht$ and~$\vht \in \Vht$ such that
\begin{subequations}
\begin{alignat*}{4}
(\vht, \zht)_{\QT} - (\dpt \uht, \zht)_{\QT} & = 0  & & \qquad \forall \zht \in \Whts,\\
(\dpt \vht, \wht)_{\QT} + (c^2 \nabla \uht, \nabla \wht)_{\QT} & = (f, \wht)_{\QT} + \langle g, \wht\rangle_{\star}
& & \qquad \forall \wht \in \Whts,
\end{alignat*}
\end{subequations}
with discrete initial conditions~$\uht(\cdot, 0) = \calR u_0$ and~$\vht(\cdot, 0) = \Pi _h v_0$, and where~$\langle\cdot, \cdot \rangle_{\star}$ denotes the duality between~$L^1(0, T; H^{-1/2}(\partial \Omega))$ and~$L^{\infty}(0, T; H^{1/2}(\partial \Omega))$.

In this case, the Ritz operator~$\calR : H^1(\Omega) \to \Vhp$ must be defined as follows:
\begin{equation*}
\begin{split}
(c^2 \nabla \calR \varphi, \nabla \zh)_{\Omega} & = (c^2 \nabla \varphi, \nabla \zh)_{\Omega} \qquad \forall \zh \in \Vhp, \\
\int_{\Omega} \calR \varphi \dx & = \int_{\Omega} \varphi \dx,
\end{split}
\end{equation*} 
where the second condition guarantees uniqueness.

Assuming, for simplicity, that~$\int_{\Omega} u_0 \dx = 0$, and proceeding as in Theorem~\ref{thm:continuous-dependence}, one can obtain the following stability estimate (cf. \cite[Thm.~4.5]{Walkington:2014}):
\begin{equation*}
\begin{split}
    \frac12 \Big(\Norm{\vht}{C^0([0, T]; L^2(\Omega))}^2 & + \Norm{c \nabla \uht}{C^0([0, T]; L^2(\Omega)^d)}^2 \Big) \\
    & \lesssim \frac12 \Big(\Norm{v_0}{L^2(\Omega)}^2 + \Norm{c \nabla u_0}{L^2(\Omega)^d}^2 \Big) \\
    & \quad + \Norm{f}{L^1(0, T; L^2(\Omega))}^2 + \Norm{g}{C^0([0, T]; H^{-1/2}(\partial \Omega))}^2 + \Norm{\dpt g}{L^1(0, T; H^{-1/2}(\partial \Omega))}^2,
\end{split}
\end{equation*}
where the terms involving~$g$ must be carefully treated. For instance, using the identity~$\Pi_{q-1}^t \vht = \dpt \uht$ (see Remark~\ref{rem:equivalent-formulations}), integration by parts in time, and the H\"older inequality, we have
\begin{equation*}
\begin{split}
\langle g, \Pi_{q-1}^t \vht \rangle_{\star} = \langle g, \dpt \uht \rangle_{\star} & = \langle g(\cdot, T), \uht(\cdot, T)\rangle - \langle g(\cdot, 0), \uht(\cdot, 0)\rangle - \langle \dpt g, \uht \rangle_{\star} \\ 
& \le \big( 2\Norm{g}{C^0([0, T]; H^{-1/2}(\partial \Omega))} + \Norm{\dpt g}{L^1(0, T; H^{-1/2}(\partial \Omega))}\big) \Norm{\uht}{C^0([0, T]; H^{1/2}(\Omega))},
\end{split}
\end{equation*}
and the term involving~$\uht$ is estimated using a continuous trace inequality, the representation formula
\begin{equation*}
\uht(\cdot, t) = \Rh u_0 + \int_0^t \dpt \uht(\cdot, s) \ds = \Rh u_0 + \int_0^t \Pi_{q-1}^t \vht(\cdot, s) \ds,
\end{equation*}
and the Poincar\'e inequality, as follows:
\begin{equation*}
\begin{split}
\Norm{\uht}{C^0([0, T]; H^{1/2}(\partial \Omega))} & \lesssim \Norm{\nabla \uht}{C^0([0, T]; L^2(\Omega)^d)}  + \Norm{\uht}{C^0([0, T]; L^2(\Omega))} \\
& \lesssim \Norm{c \nabla \uht}{C^0([0, T]; L^2(\Omega)^d)}  + \Norm{c \nabla \Rh u_0}{L^2(\Omega)^d} + T \Norm{\vht}{C^0([0, T]; L^2(\Omega))},
\end{split}
\end{equation*}
where the hidden constant now depends on the degree of approximation in time~$q$, the domain~$\Omega$, and the constants in the bound~$c_\star < c(\bx) < c^{\star}$. 
\eremk
\end{remark}

\section{Convergence analysis\label{sec:convergence-analysis}}
In this section, we derive~$(h, \tau)$-\emph{a priori} error estimates for the space--time FEM~\eqref{eq:space-time-formulation} in some~$C^0([0, T]; X)$-type norms. To do so, we first introduce some auxiliary projections in space and time in Section~\ref{sec:preliminary-a-priori}. The choice of the discrete initial conditions and the discrete liftings for nonhomogeneous Dirichlet boundary conditions is discussed in Section~\ref{sec:discrete-data}.
Our main convergence results are presented in Section~\ref{sec:a-priori}. Finally, the properties of a postprocessed approximation of~$u$ are discussed in Section~\ref{sec:postprocessing}, which are instrumental in the~\emph{a posteriori} error estimate in Section~\ref{sec:posteriori}.

\subsection{Preliminary results for the convergence analysis\label{sec:preliminary-a-priori}}
We introduce some auxiliary projections in space and time, as well as their properties that we use later in the convergence analysis.
\subsubsection{Properties of the time projection~\texorpdfstring{$\Pt$}{Pt}\label{sec:Pt}}
We first recall the definition of the auxiliary projection in~\cite[Eq.~(2.9)]{Aziz_Monk:1989}, and its stability and approximation properties.
\begin{definition}[Projection~$\Pt$]
Let~$q \in \IN$ with~$q \geq 1$. Given a partition~$\Tt$ of the time interval~$(0, T)$, the projection operator~$\Pt : H^1(0, T) \to \Vtq$ is defined for any~$w \in H^1(0, T)$ as the only element in~$\Vtq$ such that
\begin{subequations}
\begin{alignat}{3}
\label{def:Pt-1}
\Pt w(0) & = w(0),\\
\label{def:Pt-2}
((\Pt w)' - w', p_{q-1})_{\In}  & = 0 \qquad \qquad \forall p_{q-1} \in \Pp{q-1}{\In}, \text{ for } n = 1, \ldots, N.
\end{alignat}
\end{subequations}
\end{definition}

We now derive an equivalent (local) definition of~$\Pt$. Taking~$p_{q-1}$ in~\eqref{def:Pt-2} as the characteristic function~$\chi_{I_1}$ and using~\eqref{def:Pt-1}, it can be deduced that~$\Pt w(t_1) = w(t_1)$. Recursively, one can then prove that 
\begin{equation}
\label{eq:pointwise-identity-Pt}
\Pt w (\tn) = w(\tn) \quad \text{ for } n = 0, \ldots, N.
\end{equation}
Moreover, using~\eqref{eq:pointwise-identity-Pt} and integrating by parts in time in~\eqref{def:Pt-2}, we have
\begin{equation*}
\begin{split}
0 = ((\Pt w)' - w', p_{q-1})_{\In} & = (\Pt w - w)(\tn) p_{q-1}(\tn) - (\Pt(w) - w)(\tnmo) p_{q-1}(\tnmo) \\
& \quad - (\Pt w - w, p_{q-1}')_{\In} \\
& = - (\Pt w - w, p_{q-1}')_{\In},
\end{split}
\end{equation*}
whence, the projection~$\Pt$ is uniquely determined on each element~$\In \in \Tt$ as follows:
\begin{subequations}
\label{def:local-Pt}
\begin{alignat}{3}
\label{def:local-Pt-1}
\Pt w (\tnmo) & = w(\tnmo), \\
\label{def:local-Pt-2}
\Pt w (\tn) & = w(\tn), \\
\label{def:local-Pt-3}
(\Pt w - w, p_{q-2})_{\In} & = 0 \qquad \qquad \forall p_{q-2} \in \Pp{q-2}{\In}.
\end{alignat}
\end{subequations}
We now use this local definition to derive a stability bound on~$\Pt$ in the~$C^0[0, T]$ norm.

\begin{lemma}[Stability of~$\Pt$]
\label{lemma:stab-Pt}
For any~$w \in H^1(0, T)$, it holds
\begin{equation*}
\Norm{\Pt w}{C^0[0, T]} \lesssim \Norm{w}{C^0[0, T]}.
\end{equation*}
\end{lemma}
\begin{proof}
Let~$w \in H^1(0, T)$ and~$n \in \{1, \ldots, N\}$. From~\eqref{def:local-Pt-3}, we can deduce that~$\Pi_{q-2}^t \Pt w = \Pi_{q-2}^t w$, which implies
\begin{equation}
\label{eq:Pt-w-explicit}
\Pt w (t) = \Pi_{q-2}^t w (t) + \alpha L_{q-1}(t) + \beta L_q (t) \qquad \forall t \in \In,
\end{equation}
for some constants~$\alpha$ and~$\beta$. Then, using~\eqref{def:local-Pt-1} and~\eqref{def:local-Pt-2}, and the fact that the Legendre polynomials satisfy~$L_{s}(\tnmo) = (-1)^s$ and~$L_s(\tn) = 1$ for all~$s \in \IN$, it is possible to obtain the following explicit expressions for~$\alpha$ and~$\beta$:
\begin{equation}
\label{eq:alpha-beta}
\alpha = \frac{1}{2 (-1)^q} ((-1)^q \delta_n - \delta_{n - 1}) \quad \text{ and } \quad \beta = \frac{1}{2 (-1)^q} ((-1)^q \delta_n + \delta_{n - 1}),
\end{equation}
where~$\delta_{n-1} = (w - \Pi_{q-2}^t w)(\tnmo^+)$ and~$\delta_{n} = (w - \Pi_{q-2}^t w)(\tn^-)$. Inserting~\eqref{eq:alpha-beta} in~\eqref{eq:Pt-w-explicit}, and due to the uniform bound~$\Norm{L_s}{L^{\infty}(\In)} = 1$ for all~$s \in \IN$, we get
\begin{equation*}
\Norm{\Pt w}{L^{\infty}(\In)} \lesssim \Norm{\Pi_{q-2}^t w}{L^{\infty}(\In)} + \Norm{(\Id - \Pi_{q-2}^t) w}{L^{\infty}(\In)}.
\end{equation*}
The result then follows by using the~$L^{\infty}(\In)$ stability in Lemma~\ref{lemma:stab-pi-time} of~$\Pi_{q-2}^t$, the continuity of~$w$ and~$\Pt w$, and taking~$n$ as the index where the left-hand side takes its maximum value.
\end{proof}

\begin{lemma}[Estimates for~$\Pt$]
\label{lemma:estimates-Pt}
Let~$q \in \IN$ with~$q \geq 1$. Given a partition~$\Tt$ of the time interval~$(0, T)$ and~$1 \le s \le \infty$, the following estimate holds for all~$w \in W^{m + 1}_s(0, T)$ with~$1 \le m \le q$:
\begin{equation*}
\Norm{(\Id - \Pt) w}{L^s(0, T)} \lesssim \tau^{m + 1} \SemiNorm{w}{W^{m + 1}_s(0, T)}.
\end{equation*}
\end{lemma}
\begin{proof}
The result can be deduced from the following equivalent definition of~$\Pt w$ (see~\cite[Rem.~70.10]{Ern_Guermond-book-III}):
\begin{equation*}
\Pt w (t) = w(\tnmo) + \int_{\tnmo}^t \Pi_{q-1}^t w' \ds \qquad \forall t \in \In,
\end{equation*}
and the approximation properties of~$\Pi_{q-1}^t$ (see, e.g.,~\cite[Thm.~18.16(iii)]{Ern_Guermond-book-I}).
\end{proof}

\subsubsection{Properties of the Ritz projection~\texorpdfstring{$\Rh$}{Rh}}
Let~$\Ih : C^0(\overline{\Omega}) \to \Vhp$ denote the standard Lagrange interpolant, and~$\Ihg : C^0(\partial \Omega) \to \Vhp$ denote the interpolant of the restriction to~$\partial \Omega$. 
We shall denote by~$\Rh : H^1(\Omega) \cap C^0(\overline{\Omega}) \to \Vhp$ the Ritz projection taking into account Dirichlet boundary conditions. 
More precisely, for any~$\varphi \in H^1(\Omega) \cap C^0(\overline{\Omega})$, 
the projection~$\Rh \varphi \in \Ihg \varphi_{|_{\partial \Omega}} + \Vhpo$ and satisfies
\begin{equation}
\label{def:Rh-Dirichlet}
(c^2 \nabla \Rh \varphi, \nabla \zh)_{\Omega} = (c^2 \nabla \varphi, \nabla \zh)_{\Omega} \qquad \forall \zh \in \Vhpo.
\end{equation}

When~$\varphi \in H_0^1(\Omega)$, the definition of~$\Rh \varphi$ coincides with the one in~\eqref{def:Rh} for~$\Rho \varphi$. In such a case, for the sake of clarity, we shall write~$\Rho \varphi$.

We henceforth assume the following standard regularity of the spatial domain~$\Omega$.
\begin{assumption}[Elliptic regularity of~$\Omega$]
\label{asm:regularity-Omega}
The spatial domain~$\Omega \subset \IR^d$ is such that
\begin{equation*}
    \text{ if } \varphi \in H_0^1(\Omega) \ \text{ and }\ \Delta \varphi\in L^2(\Omega), \text{ then } \varphi \in H^2(\Omega),
\end{equation*}
which is valid, for instance, if~$\Omega$ is convex.
\end{assumption}

We next recall the approximation properties of the Ritz projection~$\Rh$.
\begin{lemma}[Estimates for~$\Rh$, {see, e.g., \cite[Thms. 33.2 and 33.3]{Ern_Guermond-book-II}}]
\label{lemma:estimates-Rh}
Let~$p \in \IN$ with~$p \geq 1$, and let~$\Omega$ satisfy Assumption~\ref{asm:regularity-Omega}. If~$c^2$ in~\eqref{def:Rh-Dirichlet} belongs to~$\in W_{\infty}^1(\Omega)$, the following estimates hold:
\begin{subequations}
\begin{alignat}{3}
\label{eq:estimate-Rh-H1}
\Norm{\nabla (\Id - \Rh) \varphi}{L^2(\Omega)^d} & \lesssim h^{\ell} \SemiNorm{\varphi}{H^{\ell + 1}(\Omega)} & & \qquad \forall \varphi \in H^{\ell + 1}(\Omega) \cap C^0(\overline{\Omega}),\\
\label{eq:estimate-Rh-L2}
\Norm{(\Id - \Rh) \varphi}{L^2(\Omega)} & \lesssim h^{\ell + 1} \big(\SemiNorm{\varphi}{H^{\ell + 1}(\Omega)} + \diam(\Omega)^{\frac12}|\varphi|_{H^{\ell + 1}(\partial \Omega)} \big) & & \qquad \forall \varphi \in H^{\ell + 1}(\Omega) \cap C^0(\overline{\Omega}),
\end{alignat}
\end{subequations}
with~$\max\{0, d/2 - 1\} < \ell \le p$. The hidden constants in~\eqref{eq:estimate-Rh-H1} and~\eqref{eq:estimate-Rh-L2} depend linearly on~$(c^{\star}/c_{\star})^2$, and the one in~\eqref{eq:estimate-Rh-L2} depends also on~$\SemiNorm{c^2}{W_{\infty}^1(\Omega)}$.
\end{lemma}

\begin{remark}[Interpolant~$\Ihg$]
One can reduce the regularity assumptions in Lemma~\ref{lemma:estimates-Rh} by using the Scott--Zhang interpolant operator~\cite{Scott-Zhang:1990}, which is well defined for functions in~$W_1^1(\Omega)$. This would also influence the regularity assumptions on the Dirichlet data in Theorem~\ref{thm:a-priori}.
\eremk
\end{remark}

\subsection{Discrete initial and boundary conditions\label{sec:discrete-data}}
Let~$(\uht, \vht)$ be the solution to the discrete space--time formulation~\eqref{eq:space-time-formulation}, and~$(u, v)$ be the weak solution to~\eqref{eq:model-problem-hamiltonian}.

If~$\gD \in H^2(0, T; H^s(\partial \Omega))$ with~$s > \max\{(d-1)/2, 1/2\}$, then there exists a lifting function~$u_{\gD} \in H^2(0, T; H^{s + \frac12}(\Omega))$ of~$\gD$ (i.e., such that~$u_{\gD} {}_{|_{\partial \Omega}} = \gD$), and we can set
\begin{equation}
\label{eq:discrete-liftings}
\ugh = \Pt \Ihg \gD \quad  \text{and} \quad \vgh = \Pt \Ihg \dpt \gD,
\end{equation}
as~$H^{s + \frac12}(\Omega) \hookrightarrow C^0(\overline{\Omega})$.

For nonhomogeneous Dirichlet boundary conditions, the choice in~\eqref{eq:discrete-initial} of the discrete initial conditions has to be modified as follows:
assume that~$\dpt \gD(\cdot, 0) = v_0$ on~$\partial \Omega$, then we set
\begin{equation}
\label{eq:discrete-initial-Dirichlet}
\uoh = \Rh u_0 \quad \text{ and } \quad \voh = \Ihg v_0{}_{|_{\partial \Omega}} + \Pih (v_0 - \dpt u_{\gD} {}_{|_{\SO}}),
\end{equation}
so as to ensure compatibility with the discrete liftings~$\ugh$ and~$\vgh$.

We define the following space--time projections:
\begin{equation}
\label{eq:Piht-projections}
\Piht u := \Pt \Rh u \in \ugh + \Vhto \quad \text{ and } \quad \Piht v := \Pt \Rh v \in \vgh + \Vhto,
\end{equation}
which are chosen so that the discrete error functions have zero traces. More precisely,
\begin{equation}
\label{eq:identity-discrete-errors}
\Piht u - \uht = \Pt \Rho (u - \uht) \in \Vhto \quad \text{and} \quad \Piht v - \vht = \Pt \Rho (v - \vht) \in \Vhto.
\end{equation}
Moreover, we define the projection errors
\begin{alignat*}{3}
\epiu & := u - \Piht u = (\Id - \Pt \Rh) u, \\
\epiv & := v - \Piht v = (\Id - \Pt \Rh) v.
\end{alignat*}

Finally, we define the following error functions:
\begin{equation}
\label{def:eu-ev}
\begin{split}
\eu := u - \uht = \epiu + \Piht \eu \quad \text{ and } \quad 
\ev := v - \vht = \epiv + \Piht \ev.
\end{split}
\end{equation}

\subsection{\emph{A priori} error estimates\label{sec:a-priori}}
We now present some~\emph{a priori} error bounds in Lemma~\ref{lemma:a-priori-bounds} that we use to derive some~$(h, \tau)$ error estimates in Proposition~\ref{prop:error-estimates-discrete} for the discrete errors.

\begin{assumption}[Regularity assumptions]
\label{asm:regularity-data}
For some~$s > \max\{(d - 1)/2, 1/2\}$, and~$r > \max\{1, d/2\}$, 
we assume that the data of model~\eqref{eq:model-problem-hamiltonian} satisfies (at least)
\begin{equation*}
\begin{split}
f \in L^1(0, T; L^2(\Omega)), \quad \gD \in H^2(0, T; H^s(\partial \Omega)), \quad u_0,\, v_0 \in H^r(\Omega), \\
c \in C^0(\overline{\Omega}) \cap W^1_{\infty}(\Omega) \text{ with } 0 < c_{\star} < c(\bx) < c^{\star} \text{ for all } \bx \in \Omega.
\end{split}
\end{equation*}
In addition, we require compatibility of the initial data~$(u_0, v_0)$ and the Dirichlet data~$(\gD, \dpt \gD)$ on~$\partial \Omega \times \{0\}$.
We also assume that the weak solution~$u$ to~\eqref{eq:model-problem} satisfies
\begin{equation*}
u \in H^2(0, T; H^r(\Omega)) \quad \text{and} \quad \nabla \cdot (c^2 \nabla u) \in H^1(0, T; L^2(\Omega)).
\end{equation*}
\end{assumption}
\begin{lemma}[Bounds on the discrete error]
\label{lemma:a-priori-bounds}
Let~$(u, v)$ be the weak solution to~\eqref{eq:model-problem-hamiltonian}, and~$(\uht, \vht) \in \Vht \times \Vht$ be the solution to the discrete space--time formulation~\eqref{eq:space-time-formulation} with discrete liftings~$\ugh,\, \vgh$ in~\eqref{eq:discrete-liftings} and discrete initial conditions~$\uoh,\, \voh$ defined by~\eqref{eq:discrete-initial-Dirichlet}.
Let also Assumptions~\ref{asm:regularity-Omega} and~\ref{asm:regularity-data} hold.
Then, the discrete errors in~\eqref{eq:identity-discrete-errors} satisfy the following bound:
\begin{equation}
\label{eq:bound-discrete-error}
\begin{split}
& \Norm{\Piht \ev}{C^0([0, T]; L^2(\Omega))} 
 + \Norm{c \nabla \Piht \eu}{C^0([0, T]; L^2(\Omega)^d)} \\
& \qquad \lesssim \Norm{\Pih (\Id - \Rho) \overline{v}_0}{L^2(\Omega)} + \Norm{(\Id - \Rh) \dpt v}{L^1(0, T; L^2(\Omega))} \\
& \qquad \quad + \Norm{(\Id - \Pt) (\nabla \cdot (c^2 \nabla u))}{L^1(0, T; L^2(\Omega))} + \Norm{(\Id - \Pt) (c \nabla v)}{L^1(0, T; L^2(\Omega)^d)},
\end{split}
\end{equation}
where~$\overline{v}_0 := v_0 - \dpt u_{\gD} {}_{|_{\SO}}$, and
the hidden constant depends only on~$q$.
\end{lemma}
\begin{proof}
Since the space--time formulation~\eqref{eq:space-time-formulation} is consistent, the following error equations hold:
\begin{alignat*}{3}
(c^2 \nabla \ev, \nabla \zht)_{\QT} - (c^2 \nabla \dpt \eu, \nabla \zht)_{\QT} & = 0 & & \qquad \forall \zht \in \Wht, \\
(\dpt \ev , \wht)_{\QT} + (c^2 \nabla \eu, \nabla \wht)_{\QT} & = 0 & & \qquad \forall \wht \in \Wht,
\end{alignat*}
which, splitting~$\ev$ and~$\eu$ as in~\eqref{def:eu-ev}, implies that
\begin{subequations}
\label{eq:error-equations-original}
\begin{alignat}{3}
\label{eq:error-equations-original-1}
(c^2 \nabla \Piht \ev, \nabla \zht)_{\QT} - (c^2 \nabla \dpt \Piht \eu, \nabla \zht) & = (c^2 \nabla \Piht v, \nabla \zht)_{\QT} - (c^2 \nabla \dpt \Piht u, \nabla \zht)_{\QT}, \\
\label{eq:error-equations-original-2}
(\dpt \Piht \ev, \wht)_{\QT} + (c^2 \nabla \Piht \eu, \wht)_{\QT} & = -(\dpt \epiv, \wht) - (c^2 \nabla \epiu, \nabla \wht)_{\QT}
\end{alignat}
\end{subequations}
for all~$(\zht, \wht) \in \Wht \times \Wht$.

We now simplify the terms on the right-hand side of~\eqref{eq:error-equations-original} using the orthogonality properties of~$\Pt$ and~$\Rh$ as follows:
\begin{subequations}
\begin{alignat}{3}
\nonumber
(c^2 \nabla \Piht v, \nabla \zht)_{\QT} & - (c^2 \nabla \dpt \Piht u, \nabla \zht)_{\QT}  \\
\nonumber
& = (c^2 \nabla \Pt \Rh v, \nabla \zht)_{\QT} - (c^2 \nabla \dpt \Pt \Rh u, \nabla \zht)_{\QT} \\
\nonumber
& = (c^2 \nabla \Pt v, \nabla \zht)_{\QT} - (c^2 \nabla \dpt u, \nabla \zht)_{\QT} \\
\label{eq:projection-error-term-1}
& = - (c^2 \nabla (\Id - \Pt) v, \nabla \zht)_{\QT}, \\
\nonumber
-(\dpt \epiv, \wht)_{\QT} & - (c^2 \nabla \epiu, \nabla \wht)_{\QT} \\
\nonumber
& = -(\dpt (\Id - \Pt \Rh) v, \wht)_{\QT} - (c^2\nabla (\Id - \Pt \Rh) u, \nabla \wht)_{\QT} \\
\label{eq:projection-error-term-2}
& = -( (\Id - \Rh) \dpt v, \wht)_{\QT} + ((\Id - \Pt)(\nabla \cdot (c^2 \nabla u)), \wht)_{\QT},
\end{alignat}
\end{subequations}
where, in the last step of~\eqref{eq:projection-error-term-2}, we have also integrated by parts in space.

Inserting identities~\eqref{eq:projection-error-term-1} and~\eqref{eq:projection-error-term-2} in the error equation~\eqref{eq:identity-discrete-errors}, we deduce that the discrete errors~$\Piht \ev$ and~$\Piht \eu$ solve a perturbed variational problem of the form~\eqref{eq:space-time-formulation-perturbed} with
\begin{equation*}
\begin{split}
\Piht \eu(\cdot, 0) & = \Rh (u_0 - \uoh) = 0, \\
\Piht \ev(\cdot, 0) = \Rh(v_0 - \voh) & = \Rh(v_0 - \Ihg v_0 {}_{|_{\partial \Omega}} - \Pih (v_0 - \dpt u_{\gD} {}_{|_{\SO}})) \\
& = \Rho \overline{v}_0 - \Pih \overline{v}_0 = - \Pih(\Id - \Rho) \overline{v}_0, \\
\Upsilon & = - (\Id - \Pt) v, \\
f & = - (\Id - \Rh) \dpt v + (\Id - \Pt) (\nabla \cdot (c^2 \nabla u)),
\end{split}
\end{equation*}
as~$\Pt w(0) = w(0)$.

Therefore, the error bound~\eqref{eq:bound-discrete-error} follows from applying the continuous dependence on the data in Theorem~\ref{thm:continuous-dependence}.
\end{proof}

\begin{proposition}[Estimates of the discrete error\label{prop:error-estimates-discrete}]
Let the assumptions of Lemma~\ref{lemma:a-priori-bounds} hold, and assume also that~$\overline{v}_0 \in H^{\ell + 1}(\Omega)$, $\gD \in W^2_1(0, T; H^{\ell + 1}(\partial \Omega))$, and~$u \in W^2_1(0, T; H^{\ell  + 1}(\Omega)) \cap W_1^{m + 1}(0, T; H^2(\Omega)) \cap W_1^{m + 2}(0, T; H^1(\Omega))$ for some~$1 \le \ell \le p$ and~$1 \le m \le q$. Then, the following estimate holds:
\begin{equation}
\label{eq:estimate-discrete-error}
\begin{split}
\Norm{\Piht \ev}{C^0([0, T]; L^2(\Omega))} & + 
\Norm{c \nabla \Piht \eu}{C^0([0, T]; L^2(\Omega)^d)} \\
& \lesssim h^{\ell + 1} \big(\SemiNorm{\overline{v}_0}{H^{\ell + 1}(\Omega)} + \SemiNorm{\dptt u}{L^1(0, T; H^{\ell + 1}(\Omega))}
+ \diam(\Omega)^{\frac12} \SemiNorm{\dptt \gD}{L^1(0, T; H^{\ell + 1}(\partial \Omega))} \big) \\
& \quad + \tau^{m + 1} \big(\Norm{\nabla \cdot( c^2 \nabla \dpt^{(m + 1)} u)}{L^1(0, T; L^2(\Omega))} + \Norm{c \nabla \dpt^{(m + 2)} u}{L^1(0, T; L^2(\Omega)^d)}\big),
\end{split}
\end{equation}
where the hidden constant depends only on~$(c^\star/c_\star)^2$, $\SemiNorm{c^2}{W_{\infty}^1(\Omega)}$, and~$q$.
\end{proposition}
\begin{proof}
Using the stability properties of~$\Pih$, as well as the approximation properties of~$\Pt$ and~$\Rh$ in Lemmas~\ref{lemma:estimates-Pt} and~\ref{lemma:estimates-Rh}, we can estimate each term in the~\emph{a priori} error bound in Lemma~\ref{lemma:a-priori-bounds} as follows:
\begin{alignat*}{3}
\Norm{\Pih (\Id - \Rho) \overline{v}_0}{L^2(\Omega)} & \le \Norm{(\Id - \Rho) \overline{v}_0}{L^2(\Omega)} \lesssim h^{\ell + 1} \SemiNorm{\overline{v}_0}{H^{\ell + 1}(\Omega)}, \\
\Norm{(\Id - \Rh) \dpt v}{L^1(0, T; L^2(\Omega))} & \lesssim h^{\ell + 1} \big(\SemiNorm{\dptt u}{L^1(0, T; H^{\ell + 1}(\Omega))} + \diam(\Omega)^{\frac12} \SemiNorm{\dptt \gD}{L^1(0, T; H^{\ell + 1}(\partial \Omega))} \big), \\
\Norm{(\Id - \Pt)(\nabla \cdot(c^2 \nabla u))}{L^1(0, T; L^2(\Omega))} & \lesssim \tau^{m + 1} \Norm{\nabla \cdot (c^2 \nabla \dpt^{(m + 1)} u)}{L^1(0, T; L^2(\Omega))}, \\
\Norm{(\Id - \Pt)(c \nabla v)}{L^1(0, T; L^2(\Omega)^d)} & \lesssim \tau^{m + 1} \Norm{c \nabla \dpt^{(m + 2)} u}{L^1(0, T: L^2(\Omega)^d)},
\end{alignat*}
and estimate~\eqref{eq:estimate-discrete-error} is then obtained.
\end{proof}

Using the Poincar\'e--Friedrichs inequality (see, e.g., \cite[Eq.~(5.3.3)]{Brenner-Scott:book}),
the fact that~$\Piht \eu \in \Vhto$, and Proposition~\ref{prop:error-estimates-discrete}, the following estimates for the discrete error~$\Piht \eu$ in the~$\|\cdot\|_{C^0([0, T]; L^2(\Omega))}$ norm can be obtained.
\begin{corollary}[Estimates for~$\Norm{\Piht \eu}{C^0([0, T]; L^2(\Omega))}$]
Under the assumptions of Proposition~\ref{prop:error-estimates-discrete}, the following estimates hold:
\begin{equation*}
\begin{split}
\Norm{\Piht \eu}{C^0([0, T]; L^2(\Omega))} & \lesssim \Norm{c \nabla \Piht \eu}{C^0([0, T]; L^2(\Omega)^d)} \\
& \lesssim h^{\ell + 1} \big(\SemiNorm{\overline{v}_0}{H^{\ell + 1}(\Omega)} + \SemiNorm{\dptt u}{L^1(0, T; H^{\ell + 1}(\Omega))}
+ \diam(\Omega)^{\frac12}  \SemiNorm{\dptt \gD}{L^1(0, T; H^{\ell + 1}(\partial \Omega))} \big) \\
& \quad + \tau^{m + 1} \big(\Norm{\nabla \cdot( c^2 \nabla \dpt^{(m + 1)} u)}{L^1(0, T; L^2(\Omega))} + \Norm{c \nabla \dpt^{(m + 2)} u}{L^1(0, T; L^2(\Omega)^d)}\big).
\end{split}
\end{equation*}
\end{corollary}

We are now ready to derive~$(h, \tau)$~\emph{a priori} error estimates for the error of the method. Due to the continuous embedding~$W_1^1(0, T) \hookrightarrow C^0[0, T]$ (see~\cite[Rem.~10 in~\S8.2]{Brezis:2011}), just some extra regularity on the continuous solution~$u$ is necessary in the next theorem. 
The stability properties in Lemma~\ref{lemma:stab-Pt} of~$\Pt$ are of utmost importance when studying the projection errors, which has often been neglected in previous works.
\begin{theorem}[\emph{A priori} error estimates] Let the assumptions of Proposition~\ref{prop:error-estimates-discrete} hold, and assume also that~$u \in C^{m + 2}([0, T]; L^2(\Omega))$. Then, the following estimates hold:
\label{thm:a-priori}
\begin{alignat*}{3}
\nonumber
\Norm{\ev}{C^0([0, T]; L^2(\Omega))} & \lesssim h^{\ell + 1}\big(\Norm{\dpt u}{C^0([0, T]; H^{\ell + 1}(\Omega))} + \diam(\Omega)^{\frac12} \Norm{\dpt \gD}{C^0([0, T]; H^{\ell + 1}(\Omega))} \\
\nonumber
& \qquad \qquad + \SemiNorm{\overline{v}_0}{H^{\ell + 1}(\Omega)} + \SemiNorm{\dptt u}{L^1(0, T; H^{\ell + 1}(\Omega))}
+ \diam(\Omega)^{\frac12} \SemiNorm{\dptt \gD}{L^1(0, T; H^{\ell + 1}(\partial \Omega))} \big)\\
\nonumber
& \quad + \tau^{m + 1} \big(\Norm{\dpt^{(m + 2)} u}{C^0([0, T]; L^2(\Omega))} + \Norm{\nabla \cdot( c^2 \nabla \dpt^{(m + 1)} u)}{L^1(0, T; L^2(\Omega))} \\
& \qquad \qquad  + \Norm{c \nabla \dpt^{(m + 2)} u}{L^1(0, T; L^2(\Omega)^d)}  \big), \\
\nonumber
\Norm{c \nabla \eu}{C^0([0, T]; L^2(\Omega)^d)} & \lesssim h^{\ell} \big(\Norm{u}{C^0([0, T]; H^{\ell + 1}(\Omega))} + h\SemiNorm{\overline{v}_0}{H^{\ell + 1}(\Omega)} \\
\nonumber
& \qquad \qquad + h\SemiNorm{\dptt u}{L^1(0, T; H^{\ell + 1}(\Omega))}
+ h \diam(\Omega)^{\frac12} \SemiNorm{\dptt \gD}{L^1(0, T; H^{\ell + 1}(\partial \Omega))} \big)\\
\nonumber
& \quad + \tau^{m + 1} \big( \Norm{c \nabla \dpt^{(m + 1)} u}{C^0([0, T]; L^2(\Omega)^d)} 
+ \Norm{\nabla \cdot( c^2 \nabla \dpt^{(m + 1)} u)}{L^1(0, T; L^2(\Omega))} \\
& \qquad \qquad + \Norm{c \nabla \dpt^{(m + 2)} u}{L^1(0, T; L^2(\Omega)^d)} \big), \\
\nonumber
\Norm{\eu}{C^0([0, T]; L^2(\Omega)^d)} & \lesssim h^{\ell + 1} \big( \Norm{u}{C^0([0, T]; H^{\ell + 1}(\Omega))} + \diam(\Omega)^{\frac12} \Norm{\gD}{C^0([0, T]; H^{\ell + 1}(\Omega))} \\
\nonumber
& \qquad \qquad + \SemiNorm{\overline{v}_0}{H^{\ell + 1}(\Omega)} + \SemiNorm{\dptt u}{L^1(0, T; H^{\ell + 1}(\Omega))}
+ \diam(\Omega)^{\frac12} \SemiNorm{\dptt \gD}{L^1(0, T; H^{\ell + 1}(\partial \Omega))}\big)\\
\nonumber
& \quad + \tau^{m + 1} \big( \Norm{\dpt^{(m + 1)} u}{C^0([0, T]; L^2(\Omega))} + \Norm{\nabla \cdot( c^2 \nabla \dpt^{(m + 1)} u)}{L^1(0, T; L^2(\Omega))}  \\
& \qquad \qquad + \Norm{c \nabla \dpt^{(m + 2)} u}{L^1(0, T; L^2(\Omega)^d)}\big).
\end{alignat*}
\end{theorem}
\begin{proof}
We first estimate the projection errors. 

Using the triangle inequality, the stability bound in Lemma~\ref{lemma:stab-Pt} for~$\Pt$, and the approximation properties of~$\Pt$ and~$\Rh$ in Lemmas~\ref{lemma:estimates-Pt} and~\ref{lemma:estimates-Rh}, we have
\begin{alignat}{3}
\nonumber
\Norm{\epiv}{C^0([0, T]; L^2(\Omega))} & \le 
\Norm{(\Id - \Pt) v}{C^0([0, T]; L^2(\Omega))} + \Norm{\Pt(\Id - \Rh) v}{C^0([0, T]; L^2(\Omega))} \\
\nonumber
& \lesssim \Norm{(\Id - \Pt) v}{C^0([0, T]; L^2(\Omega))} + \Norm{(\Id - \Rh) v}{C^0([0, T]; L^2(\Omega))} 
\\
\nonumber
& \lesssim \tau^{m + 1} \Norm{\dpt^{(m + 2)} u}{C^0([0, T]; L^2(\Omega))} + h^{\ell + 1}\big( \Norm{v}{C^0([0, T]; H^{\ell + 1}(\Omega))} \\
\label{eq:estimate-epiv}
& \quad + \diam(\Omega)^{\frac12} \Norm{\dpt \gD}{C^0([0, T]; H^{\ell + 1}(\partial \Omega))} \big).
\end{alignat}

As for~$c \nabla \epiu$, we use the triangle 
inequality, the stability properties in Lemma~\ref{lemma:stab-Pt} of~$\Pt$, and the approximation properties of~$\Pt$ and~$\Rh$ in Lemmas~\ref{lemma:estimates-Pt} and~\ref{lemma:estimates-Rh} to obtain
\begin{alignat}{3}
\nonumber
\Norm{c \nabla \epiu }{C^0([0, T]; L^2(\Omega)^d)} 
& \le \Norm{(\Id - \Pt) (c \nabla u)}{C^0([0, T]; L^2(\Omega)^d)} + \Norm{\Pt ( c \nabla (u - \Rh u))}{C^0([0, T]; L^2(\Omega)^d)} \\
\nonumber
& \le \Norm{(\Id - \Pt)(c\nabla u)}{C^0([0, T]; L^2(\Omega)^d)} + \Norm{c \nabla (\Id - \Rh) u}{C^0([0, T]; L^2(\Omega)^d)} \\
\label{eq:estimate-grad-epiu}
& \lesssim \tau^{m+1} \Norm{c \nabla \dpt^{(m + 1)} u}{C^0([0, T]; L^2(\Omega)^d)} + h^{\ell} \Norm{u}{C^0([0, T]; H^{\ell + 1}(\Omega))}.
\end{alignat}

Moreover, the same steps in~\eqref{eq:estimate-epiv} can be used to deduce that
\begin{equation}
\begin{split}
\Norm{\epiu}{C^0([0, T]; L^2(\Omega))} & \lesssim \tau^{m + 1} \Norm{\dpt^{(m + 1)} u}{C^0([0, T]; L^2(\Omega))} + h^{\ell + 1}\big( \Norm{u}{C^0([0, T]; H^{\ell + 1}(\Omega))} \\
\label{eq:estimate-epiu}
& \quad + \diam(\Omega)^{\frac12} \Norm{\gD}{C^0([0, T]; H^{\ell + 1}(\partial \Omega))} \big).
\end{split}
\end{equation}

Finally, using the error splitting in~\eqref{def:eu-ev}, the triangle inequality, estimates~\eqref{eq:estimate-epiv}, \eqref{eq:estimate-grad-epiu}, and~\eqref{eq:estimate-epiu}, and Proposition~\ref{prop:error-estimates-discrete}, the desired result can be easily obtained.
\end{proof}

\begin{remark}[$(p, q)$-explicit estimates]
In this work, we have focused on the~$h$-convergence of the method. Deriving stability bounds and a priori error estimates that are also explicit in the degrees of approximation in space~$(p)$ and time~$(q)$ is also of great interest (see, e.g., \cite{Schotzau_Schwab:2000,Cangiani_Dong_Georgoulis:2017,Antonietti_Mazzieri_Migliorini:2020}). However, in order to avoid suboptimal error estimates,  some milder regularity assumptions are needed. For instance, in the current setting, the constant~$\CS$ in Proposition~\ref{prop:weak-continuous-wave} behaves like~$\CS(q) \lesssim q$ (see~\cite[Eq.~(3.6.4)]{Schwab-book:1998}), eventually leading to a suboptimal estimate. 

To circumvent this, one may bound the term~$(f, \Pi_{q-1}^t \vht)_{\QT}$ different, as follows:
\begin{equation*}
(f, \Pi_{q-1}^t \vht)_{\QT} \le \Norm{f}{L^2(\QT)} \Norm{\Pi_{q-1}^t \vht}{L^2(\QT)} \le \sqrt{T} \Norm{f}{L^2(\QT)} \Norm{\vht}{L^{\infty}(0, T; L^2(\Omega))},
\end{equation*}
which requires~$f \in L^2(\QT)$ and introduces an additional dependence on~$\sqrt{T}$. 

An analogous situation was highlighted in~\cite[Rem.~1]{Dong-Mascotto-Wang:2024}, where a~$q$-explicit analysis of the dG--cG method for the wave equation was carried out.
\eremk
\end{remark}

\subsection{Postprocessed approximation~\texorpdfstring{$\uhtsup$}{uht}\label{sec:postprocessing}}
Taking advantage of the quasi-optimal convergence of~$\vht$, one can define a postprocessed approximation~$\uhtsup \in \mathcal{V}_{h, \tau}^{p, q+ 1}$ as follows:
\begin{equation}
\label{def:uhtsup}
\uhtsup(\cdot, t) = \uht(\cdot, 0) + \int_0^t \vht(\cdot, s) \ds \qquad \forall \bx \in \Omega.
\end{equation}

We consider the case of homogeneous Dirichlet boundary conditions so that~\eqref{eq:space-time-equivalence-L2}, which implies~$\Pi_{q-1}^t \vht = \dpt \uht$, holds. We summarize some important properties of~$\uhtsup$.
\begin{enumerate}[label = (\alph*), ref = (\alph*)]
\item \label{Prop-a} $\uhtsup(\cdot, 0) = \uht(\cdot, 0)$;
\item \label{Prop-b} since~$\Pi_{q-1}^t \vht = \dpt \uht$, for~$n = 1, \ldots, N$, it holds
\begin{equation*}
    \uhtsup(\cdot, \tn) = \uht(\cdot, 0) + \sum_{m = 1}^{n - 1} \int_{\tm}^{t_{m+1}} \vht(\cdot, s) \ds = \uht(\cdot, 0) + \sum_{m = 1}^{n - 1} \int_{\tm}^{t_{m+1}} \dpt \uht(\cdot, s) \ds = \uht(\cdot, \tn);
\end{equation*}
\item \label{Prop-c} $\dpt \uhtsup = \vht$.
\end{enumerate}

We now use the above properties to prove some bounds that we use in Section~\ref{sec:posteriori}.
\begin{lemma}
\label{lemma:uhtsup-uht}
Let~$(\uht, \vht) \in \Vhto \times \Vhto$ be the solution to the discrete space--time formulation~\eqref{eq:space-time-formulation} with homogeneous Dirichlet boundary conditions~$(\gD = 0)$, and~$\uhtsup \in \Vhtqpo$ be given by~\eqref{def:uhtsup}. Then, for~$n =1, \ldots, N$, it holds 
\begin{subequations}
\begin{alignat}{3}
\label{eq:ort-uht-uhtsup}
\text{(if~$q > 1$)} \quad (\uhtsup - \uht, \wht)_{\Qn} & = 0 & & \quad \forall \wht \in \Pp{q-2}{\In} \otimes \Vhpo, \\
\label{eq:bound-Linf-uhtsup-uht}
\Norm{\uhtsup - \uht}{L^{\infty}(\In; L^2(\Omega))} & \le \sqrt{C_{q, \star} \tau_n} \Norm{(\Id - \Pi_{q-1}^t ) \vht}{L^2(\Qn)}, \\
\label{eq:bound-L1-uhtsup-uht}
\Norm{\uhtsup - \uht}{L^1(\In; L^2(\Omega))} & \le \tau_n \Norm{(\Id - \Pi_{q-1}^t) \vht}{L^1(\In; L^2(\Omega))},
\end{alignat}
\end{subequations}
where
\begin{equation}
\label{def:Cqstar}
\displaystyle C_{q, \star} := \begin{cases}
\displaystyle \frac{1}{\pi} & \text{ if } q = 1, \\[1em]
\displaystyle \frac{1}{2\sqrt{(q-1)q}} & \text{ if } q > 1.
\end{cases}
\end{equation}
\end{lemma}
\begin{proof}
Let~$\zht \in \Pp{q-1}{\In} \otimes \Vhpo$ and~$n \in \{1, \ldots, N\}$.
If~$q > 1$, the following identity follows from  properties~\ref{Prop-b} and~\ref{Prop-c}, and integration by parts in time:
\begin{equation*}
0 = ((\Id - \Pi_{q-1}^t) \vht, \zht)_{\Qn} = (\dpt (\uhtsup - \uht), \zht)_{\Qn} = - (\uhtsup - \uht, \dpt \zht)_{\Qn},
\end{equation*}
which implies~\eqref{eq:ort-uht-uhtsup}.

Using the Nagy inequality in~\cite[Eq.~(1)]{Nagy:1941}, and the fact that~$(\uhtsup - \uht)(\cdot, t_m) = 0$ for~$m \in \{n-1, n\}$, the following bound can be obtained:
\begin{equation}
\label{eq:Nagy}
\begin{split}
\Norm{\uhtsup - \uht}{L^{\infty}(\In; L^2(\Omega))}^2 & \le \Norm{\uhtsup - \uht}{L^2(\Qn)} \Norm{\dpt(\uhtsup - \uht)}{L^2(\Qn)} \\
& = \Norm{\uhtsup - \uht}{L^2(\Qn)} \Norm{(\Id - \Pi_{q-1}^t) \vht}{L^2(\Qn)}.
\end{split}
\end{equation}

We now distinguish two cases.
\paragraph{Case~$q = 1$.} Bound~\eqref{eq:bound-Linf-uhtsup-uht} can be obtained  from~\eqref{eq:Nagy} and the following Poincar\'e inequality in~1D (see~\cite[\S2]{Payne_Weinberger:1960}):
\begin{equation*}
\Norm{\uhtsup - \uht}{L^2(\Qn)} \le \frac{\tau_n}{\pi} \Norm{(\Id - \Pi_{q-1}^t) \vht}{L^2(\Qn)}.
\end{equation*}

\paragraph{Case~$q > 1$.} Using the orthogonality property~\eqref{eq:ort-uht-uhtsup}, the Cauchy--Schwarz inequality, and the approximation property of~$\Pi_{q-2}^t$
\begin{alignat}{3}
\nonumber
\Norm{\uhtsup - \uht}{L^2(\Qn)}^2 & = (\uhtsup - \uht, \uhtsup - \uht)_{\Qn} \\
\nonumber
& = (\uhtsup - \uht, (\Id - \Pi_{q-2}^t) (\uhtsup - \uht))_{\Qn} \\ 
\nonumber
&\le \Norm{\uhtsup - \uht}{L^2(\Qn)} \Norm{(\Id - \Pi_{q-2}^t) (\uhtsup - \uht)}{L^2(\Qn)} \\
\nonumber
& \le C_{q, \star}\tau_n \Norm{\uhtsup - \uht}{L^2(\Qn)} \Norm{\dpt (\uhtsup - \uht)}{L^2(\Qn)} \\
\label{eq:aux-Linf-q2}
& = C_{q, \star} \tau_n \Norm{\uhtsup - \uht}{L^2(\Qn)} \Norm{(\Id - \Pi_{q-1}) \vht}{L^2(\Qn)},
\end{alignat}
with~$C_{q, \star}$ as in~\eqref{def:Cqstar} (see~\cite[Thm.~3.11]{Schwab-book:1998}). Combining~\eqref{eq:aux-Linf-q2} with~\eqref{eq:Nagy}, we get~\eqref{eq:bound-Linf-uhtsup-uht} for the case~$q > 1$.

As a consequence of properties~\ref{Prop-b}  and~\ref{Prop-c}, $\uhtsup$ can be characterized locally on each time slab~$\Qn$ as follows
\begin{alignat*}{3}
\uhtsup(\cdot, t) = \uht(\cdot, \tn) - \int_t^{\tn} \vht(\cdot, \tn) \ds,
\end{alignat*}
so one can deduce~\eqref{eq:bound-L1-uhtsup-uht} using the H\"older inequality. 
\end{proof}

\begin{remark}[Superconvergence of~$\uhtsup$]
\label{rem:superconvergence}
Using the identity~$u(\cdot, t) = u(\cdot, \tn) - \int_{t}^{\tn} v(\cdot, s) \ds$ for all~$t \in \In$, properties~\ref{Prop-b} and~\ref{Prop-c}, and the H\"older inequality, we have
\begin{alignat*}{3}
\Norm{(u - \uhtsup)(\cdot, t)}{L^2(\Omega)}
& \le \Norm{u - \uht}{L^2(\Sn)} + \tau_n \Norm{v - \vht}{C^0([\tnmo, \tn]; L^2(\Omega))} \qquad  \forall t \in \In.
\end{alignat*}
Thus, by Theorem~\ref{thm:a-priori} and  the well-known superconvergence at the discrete times~$\tn$ for~$q > 1$ (see~\cite[Thm.~8]{French_Peterson:1996}), we deduce that
\begin{alignat*}{3}
\Norm{u - \uhtsup}{C^0([0, T]; L^2(\Omega))} & \le \max_{n \in \{1, \ldots, N\}} \Norm{u - \uht}{L^2(\Sn)} + \tau \Norm{v - \vht}{C^0([0,T]; L^2(\Omega))} \\
& \lesssim \tau^{2q} + h^{p + 1} + \tau^{q + 2} + \tau h^{p + 1},
\end{alignat*}
where the leading term is of order~$\mathcal{O}(\tau^{q + 2} + h^{p + 1})$. Therefore, $\uhtsup$ superconverges with an extra order in time for~$q > 1$. 

In~\cite[\S3.2]{Bause_etal:2020}, some lifting reconstructions of~$\uht$ and~$\vht$ for the underintegrated version of~\eqref{eq:space-time-formulation} are used to obtain superconvergent approximations of~$u$ and~$v$. Such liftings are computed independently for~$\uht$ and~$\vht$ by means of some correction terms using the values at the Gauss-Lobato points.
On the other hand, the construction of~$\uhtsup$ combines information of~$\uht$ and~$\vht$, which is somehow related to the choice of~$(\Piht u, \Piht v)$ in the error analysis in~\cite[Eqs.~(3.5)--(3.7)]{Karakashian_Makridakis:2005}.
\eremk
\end{remark}
The postprocessed approximation~$\uhtsup$ will also play an important role in the~\emph{a posteriori} error estimate in next section. We emphasize that the convergence properties of~$\uhtsup$ are not used in Section~\ref{sec:posteriori}, but only semidiscrete-in-time versions of the properties derived in Lemma~\ref{lemma:uhtsup-uht}. 
\section{\emph{A posteriori} error estimates for the semidiscrete-in-time scheme\label{sec:posteriori}}
In the spirit of~\cite[\S3]{Dong-Mascotto-Wang:2024}, we now derive an \emph{a posteriori} error estimate for the semidiscrete-in-time scheme.

For the sake of simplicity, along this section, we assume that the spatial domain~$\Omega$ satisfies Assumption~\ref{asm:regularity-Omega}, and
\begin{equation}
\label{eq:simple-regularity-assumptions}
\gD = 0, \quad c \in \IR^+, \quad u_0 \in H^2(\Omega) \cap H_0^1(\Omega), \quad v_0 \in H_0^1(\Omega), \quad f \in H^1(0, T; L^2(\Omega)).
\end{equation}
Then, the continuous weak solution~$u$ to~\eqref{eq:model-problem} satisfies (see~\cite[Thm.~5 in~\S7.2.3]{Evans:2010} and~\cite[Prop.~3.1]{Dong-Mascotto-Wang:2024})
\begin{equation}
\label{eq:strong-regularity-u}
u \in L^{\infty}(0, T; H^2(\Omega)) \cap W^1_{\infty}(0, T; H_0^1(\Omega)) \cap W^2_{\infty}(0, T; L^2(\Omega)) \cap W^3_{\infty}(0, T; H^{-1}(\Omega)).
\end{equation}

We define the following auxiliary spaces:
\begin{equation}
\begin{split}
\VVtq & := \{\varphi \in C^0([0, T]; H_0^1(\Omega))\, :\, \varphi(\cdot, 0) = 0, \text{ and } \varphi_{|_{\Qn}} \in \Pp{q}{\In} \otimes H_0^1(\Omega), \text{ for } n = 1, \ldots, N\}, \\
\WWtqmo & := \{\varphi \in L^2(0, T; H_0^1(\Omega))\, :\, \varphi_{|_{\Qn}} \in \Pp{q - 1}{\In} \otimes H_0^1(\Omega), \text{ for } n = 1, \ldots, N\},
\end{split}
\end{equation}
and consider the following semidiscretization in time of the Hamiltonian formulation~\eqref{eq:model-problem-hamiltonian}: find~$\ut \in u_0 + \VVtq$ and~$\vt \in v_0 + \VVtq$, such that
\begin{subequations}
\label{eq:semi-discrete-in-time}
\begin{alignat}{3}
(\vt, \zt)_{\QT} - (\dpt \ut, \zt)_{\QT} & = 0 & & \qquad \forall \zt \in \WWtqmo, \\
(\dpt \vt, \wt)_{\QT} - c^2 (\Delta \ut, \wt)_{\QT} & = (f, \wt)_{\QT} & & \qquad \forall \wt \in \WWtqmo.
\end{alignat}
\end{subequations}

Let also~$\utsup$ denote the semidiscrete version of the postprocessed variable in Section~\ref{sec:postprocessing}, whose properties~\ref{Prop-a}, \ref{Prop-b}, and~\ref{Prop-c} play a key role in the proof of next theorem.

\begin{theorem}[\emph{A posteriori} error estimate]
\label{thm:a-posteriori}
Let the regularity assumptions in~\eqref{eq:simple-regularity-assumptions} hold. Let also~$(u, v)$ be the continuous weak solution to~\eqref{eq:model-problem-hamiltonian}, and~$(\ut, \vt)$ be the solution to the semidiscrete-in-time formulation~\eqref{eq:semi-discrete-in-time}.
Then, the following estimate holds:
\begin{alignat*}{3}
& \Norm{u - \ut}{C^0([0, T]; L^2(\Omega))} \\
& \qquad \le \max_{n \in \{1, \ldots, N\}} \sqrt{C_{q, \star} \tau_n} \Norm{(\Id - \Pi_{q-1}^t) \vt}{L^2(\In; L^2(\Omega))} \\
& \quad \qquad + 2 \Cp(q-1) \sum_{n = 1}^{m - 1} \tau_n \Norm{(\Id - \Pi_{q-1}^t) f}{L^1(\In; L^2(\Omega))} + 2 \tau_m \Norm{(\Id - \Pi_{q-1}^t) f}{L^1(I_m; L^2(\Omega))} \\
& \quad \qquad + 2 c^2 C_{\bullet}(q)  \sum_{n = 1}^{m = 1} \tau_n  \Norm{(\Id - \Pi_{q-1}^t) \Delta \vt}{L^1(\In; L^2(\Omega))} 
+ 2 c^2 \tau_m \Norm{(\Id - \Pi_{q-1}^t) \Delta \vt}{L^1(I_m; L^2(\Omega))} \\
& \quad \qquad + 2 c^2 \Cp(q-1) \sum_{n = 1}^{m - 1} \tau_n \Norm{(\Id - \Pi_{q-1}^t ) \Delta \ut}{L^1(\In; L^2(\Omega))} + 2 c^2 \tau_m \Norm{(\Id - \Pi_{q-1}^t) \Delta \ut}{L^1(I_m; L^2(\Omega))}, 
\end{alignat*}
where~$I_m$ is the interval where~$\Norm{(\utsup - \ut)(\cdot, t)}{L^2(\Omega))}$ takes its maximum, $C_{q, \star}$ as in~\eqref{def:Cqstar}, and
$$\Cp(s) := \begin{cases}
\pi^{-\frac12} & \text{ if } s = 0, 1, 2,\\[0.8em]
\frac{1}{(s - 2)\pi} & \text{ if } s \geq 3,
\end{cases} \qquad \qquad 
C_{\bullet}(q) := \begin{cases}
|\tm - \tnmo| & \text{ if } q = 1, \\[0.8em]
\frac{\Cp(q - 2) \tau_n}{2} & \text{ if } q \geq 2.
\end{cases}$$
\end{theorem}
\begin{proof}
Since~$u - \utsup$ is continuous in time, there exists~$\xi \in \overline{I_m}$ for some~$m \in \{1, \ldots, N\}$ 
such that
\begin{equation}
\label{def:xi}
\Norm{u - \utsup}{C^0([0, T]; L^2(\Omega))} = \Norm{(u - \utsup)(\cdot, \xi)}{L^2(\Omega)}.
\end{equation}
Then, we define the following auxiliary function:
\begin{equation}
\label{def:wt}
\wt(\cdot, t) = \int_t^{\xi} (u - \utsup)(\cdot, s) \quad \text{ in } L^2(\Omega),
\end{equation}
which satisfies the following three important properties:
\begin{subequations}
\begin{alignat}{3}
\label{eq:property-wt-1}
\wt(\cdot, \xi) & = 0, \\
\label{eq:property-wt-2}
\wt(\cdot, 0) & = \int_0^{\xi} (u - \utsup)(\cdot, s) & & \quad \text{ in } L^2(\Omega), \\
\label{eq:property-wt-3}
\dpt \wt(\cdot, t) & = - (u - \utsup)(\cdot, t) & & \quad \text{ in } L^2(\Omega).
\end{alignat}
\end{subequations}

Furthermore, due to the identity~$z(t) = z(\xi) - \int_t^{\xi} z' \ds$ and property~\eqref{eq:property-wt-1}, for~$n = 1, \ldots, m-1$, it holds
\begin{alignat}{3}
\label{eq:bound-wt}
\Norm{\wt}{L^{\infty}(\In; L^2(\Omega))} \le |\xi - \tnmo| \Norm{\dpt \wt}{L^{\infty}(\In; L^2(\Omega))} \le |t_m - \tnmo| \Norm{u - \utsup}{L^{\infty}(0, T; L^2(\Omega))},
\end{alignat}
and
\begin{equation}
\label{eq:bound-wt-m}
\Norm{\wt}{L^{\infty}(t_{m-1}, \xi; L^2(\Omega))} \le \Norm{\wt}{L^{\infty}(I_m; L^2(\Omega))} \le \tau_m \Norm{\dpt \wt}{L^{\infty}(I_m; L^2(\Omega))}.
\end{equation}

By the regularity of~$u$ in~\eqref{eq:strong-regularity-u}, and the semidiscrete-in-time formulation~\eqref{eq:semi-discrete-in-time}, the following error equations hold in the sense of~$L^2(0, T; L^2(\Omega))$
\begin{subequations}
\label{eq:pointwise-errors}
\begin{alignat}{3}
\label{eq:pointwise-errors-1}
\Pi_{q-1}^t \vt - \dpt \ut & = 0, \\
\label{eq:pointwise-errors-2}
\dpt (v - \vt) - c^2 \Delta(u - \ut) & = (\Id - \Pi_{q-1}^t )f + c^2 \Delta (\Id - \Pi_{q-1}^t) \ut.
\end{alignat}
\end{subequations}
Multiplying~\eqref{eq:pointwise-errors-2} by~$\wt$ in~\eqref{def:wt}, adding and subtracting suitable terms, and integrating in time from~$0$ to~$\xi$, we get
\begin{equation}
\label{eq:integrated-identity-posteriori}
\begin{split}
& \int_0^{\xi} \big(\dpt(v - \vt), \wt)_{\Omega} + c^2 (\nabla (u - \utsup), \nabla \wt)_{\Omega}\big) \ds \\
& \qquad = \int_{0}^{\xi} \big((\Id - \Pi_{q-1}^t f, \wt)_{\Omega} + c^2 (\Delta (\utsup - \ut), \wt)_{\Omega} + c^2 ((\Id - \Pi_{q-1}^t) \Delta \ut, \wt)_{\Omega}\big)\ds.
\end{split}
\end{equation}
We consider each term separately.

Integration by parts in time gives
\begin{equation}
\label{eq:aux-identity-(v-vt)}
\begin{split}
\int_0^{\xi} (\dpt(v - \vt), \wt)_{\Omega} \ds & = \big( (v - \vt)(\cdot, \xi), \wt(\cdot, \xi)\big)_{\Omega} - \big((v - \vt)(\cdot, 0), \wt(\cdot, 0) \big)_{\Omega} \\
& \quad - \int_0^{\xi} (v - \vt, \dpt \wt)_{\Omega}\, \ds.
\end{split}
\end{equation}
The first term on the right-hand side of~\eqref{eq:aux-identity-(v-vt)} vanishes due to property~\eqref{eq:property-wt-1} of~$\wt$, whereas the second term vanishes since~$v(\cdot, 0) = \vt(\cdot, 0) = v_0$ in~$L^2(\Omega)$.
As for the third term, we can use property~\eqref{eq:property-wt-3} of~$\wt$, the definition of~$\xi$ in~\eqref{def:xi}, and the fact that~$\utsup(\cdot, 0) = \ut(\cdot, 0) = u_0$ to obtain
\begin{alignat}{3}
\nonumber
-\int_0^{\xi} (v - \vt, \dpt \wt)_{\Omega}\, \ds & = \int_0^{\xi} (v - \vt, u - \utsup)_{\Omega}\, \ds \\
\nonumber
& = \int_{0}^{\xi} (\dpt(u - \utsup), u - \utsup)_{\Omega}\, \ds \\
\nonumber
& = \frac12 \Big(\Norm{(u - \utsup)(\cdot, \xi)}{L^2(\Omega)}^2 - \Norm{(u - \utsup)(\cdot, 0)}{L^2(\Omega)}^2\Big) \\
\label{eq:aux-posteriori-1}
& = \frac12 \Norm{u - \utsup}{C^0([0, T]; L^2(\Omega))}.
\end{alignat}

Using property~\eqref{eq:property-wt-3} of~$\wt$, the fact that~$\dpt \utsup = \vt$ and~$v = \dpt u$, and integrating in time, we obtain
\begin{alignat}{3}
\nonumber
c^2 \int_0^{\xi} (\nabla (u - \utsup), \nabla \wt)_{\Omega} \ds & = -c^2 \int_0^{\xi} (\nabla \dpt \wt, \nabla \wt)_{\Omega} \ds \\
\label{eq:aux-posteriori-2}
& = -\frac{c^2}{2} \Big( \Norm{\nabla \wt(\cdot, \xi)}{L^2(\Omega)^d}^2 - \Norm{\nabla \wt(\cdot, 0)}{L^2(\Omega)^d}^2\Big),
\end{alignat}
where the first term on the right-hand side vanishes due to property~\eqref{eq:property-wt-1} of~$\wt$.

Combining identities~\eqref{eq:aux-posteriori-1} and~\eqref{eq:aux-posteriori-2} with~\eqref{eq:integrated-identity-posteriori}, we get
\begin{alignat}{3}
\nonumber
& \frac12 \Norm{u - \utsup}{C^0([0, T]; L^2(\Omega))}^2 + \frac{c^2}{2} \Norm{\nabla \wt(\cdot, 0)}{L^2(\Omega)^d}^2 \\
\nonumber
& \quad = \int_0^{\xi} \Big(((\Id - \Pi_{q-1}^t) f, \wt)_{\Omega} + c^2 (\Delta ( \utsup - \ut), \wt)_{\Omega} + c^2 ((\Id - \Pi_{q-1}^t) \Delta \ut, \wt)_{\Omega}\Big) \ds \\
\label{eq:final-identity-posteriori}
& \quad =: R_1 + R_2 + R_3.
\end{alignat}
We now bound the terms~$\{R_i\}_{i = 1}^3$.

Let~$\zeta_{\tau} \in \WWtqmo$ be such that, on each time slab~$\Qn$, $\zeta_{\tau}(\bx, \cdot)$ is the best polynomial in time approximation of~$\wt(\bx, \cdot)$ in~$L^{\infty}(\In)$ for a.e.~$\bx \in \Omega$ and~$n = 1, \ldots, N$. Using the H\"older inequality, the orthogonality and approximation properties of~$\Pi_{q-1}^t$, and bounds~\eqref{eq:bound-wt} and~\eqref{eq:bound-wt-m}, for any function~$\omega_{\tau} \in \WWtqmo$, we have
\begin{alignat*}{3}
R_1 = \int_0^{\xi} ((\Id - \Pi_{q-1}^t) f, \wt)_{\Omega}\, \ds & = \sum_{n = 1}^{m-1} \int_{\tnmo}^{\tn} ((\Id - \Pi_{q-1}^t) f, \wt - \zeta_{\tau})_{\Omega} + \int_{t_{m-1}}^{\xi} ((\Id - \Pi_{q-1}^t) f, \wt)_{\Omega} \\
& \le \Cp(q-1) \sum_{n = 1}^{m - 1} \tau_n \Norm{(\Id - \Pi_{q-1}^t) f}{L^1(\In; L^2(\Omega))} \Norm{\dpt \wt}{L^{\infty}(\In; L^2(\Omega))} \\
& \quad + \tau_m\Norm{(\Id - \Pi_{q-1}^t) f}{L^1(I_m; L^2(\Omega))} \Norm{\dpt \wt}{L^{\infty}(I_m; L^2(\Omega))},
\end{alignat*}
where the explicit expression of the constant~$\Cp$ of the best polynomial approximation in~$L^{\infty}(\In)$ can be found, for instance, in~\cite[Lemma~3.4]{Dong-Mascotto-Wang:2024}.

As for~$R_2$, we distinguish two cases.
\paragraph{Case~$q = 1$.} 
We can use~\eqref{eq:bound-L1-uhtsup-uht} together with the H\"older inequality, and bounds~\eqref{eq:bound-wt} and~\eqref{eq:bound-wt-m} to obtain
\begin{alignat*}{3}
R_2 & = c^2 \int_0^{\xi} (\Delta (\utsup - \ut), \wt)_{\Omega}\, \ds \\
& \le c^2 \sum_{n = 1}^{m - 1} \Norm{\Delta(\utsup - \ut)}{L^1(\In; L^2(\Omega))} \Norm{\wt}{L^{\infty}(\In; L^2(\Omega))} \\
& \quad + c^2 \Norm{\Delta(\utsup - \ut)}{L^1(I_m; L^2(\Omega))} \Norm{\wt}{L^{\infty}(t_{m - 1}, \xi; L^2(\Omega))} \\
& \le c^2 \sum_{n = 1}^{m - 1} |t_m - \tnmo| \tau_n \Norm{(\Id - \Pi_{q-1}^t) \Delta \vt}{L^1(\In; L^2(\Omega))}\Norm{\dpt \wt}{L^{\infty}(0, T; L^2(\Omega))} \\
& \quad + c^2 \tau_m^2 \Norm{(\Id - \Pi_{q-1}^t) \Delta \vt}{L^1(I_m; L^2(\Omega))} \Norm{\dpt \wt}{L^{\infty}(0, T; L^2(\Omega))}.
\end{alignat*}

\paragraph{Case~$q > 1$.} 
Let~$\omega_{\tau} \in \WWtqmt$ be such that, on each time slab~$\Qn$, $\omega_{\tau}(\bx, \cdot)$ is the best polynomial in time approximation of~$\wt(\bx, \cdot)$ in~$L^{\infty}(\In)$ for a.e.~$\bx \in \Omega$ and~$n = 1, \ldots, N$.
The orthogonality properties of~$(\utsup - \ut)$ for~$q > 1$ (see~\eqref{eq:ort-uht-uhtsup}) allows us to use the Poincar\'e inequality in~$L^1(\In)$ with constant~$1/2$ (see~\cite[Thm.~3.1]{Acosta_Duran:2004}), which, combined with the H\"older inequality, and bounds~\eqref{eq:bound-wt} and~\eqref{eq:bound-wt-m}, leads to the following bound:
\begin{alignat*}{3}
R_2 = c^2 \int_0^{\xi} (\Delta (\utsup - \ut), \wt)_{\Omega}\, \ds 
& = c^2 \sum_{n = 1}^{m - 1} (\Delta (\utsup - \ut), \wt - \omega_{\tau})_{\Qn} + c^2 \int_{\tmmo}^{\xi} (\Delta(\utsup - \ut), \wt)_{\Omega}\, \ds \\
& \le c^2 \sum_{n = 1}^{m - 1} \Norm{\Delta(\utsup - \ut)}{L^1(\In; L^2(\Omega))} \Norm{\wt - \omega_{\tau}}{L^{\infty}(\In; L^2(\Omega))} \\
& \quad + c^2 \Norm{\Delta(\utsup - \ut)}{L^1(I_m; L^2(\Omega))} \Norm{\wt}{L^{\infty}(t_{m - 1}, \xi; L^2(\Omega))} \\
& \le \frac{c^2 \Cp(q-2)}{2} \sum_{n = 1}^{m - 1} \tau_n^2 \Norm{(\Id - \Pi_{q-1}^t) \Delta \vt}{L^1(\In; L^2(\Omega))}\Norm{\dpt \wt}{L^{\infty}(0, T; L^2(\Omega))} \\
& \quad + \frac{c^2 \tau_m^2}{2} \Norm{(\Id - \Pi_{q-1}^t) \Delta \vt}{L^1(I_m; L^2(\Omega))} \Norm{\dpt \wt}{L^{\infty}(0, T; L^2(\Omega))}.
\end{alignat*}

The term~$R_3$ can be bounded similarly to~$R_1$, as follows:
\begin{alignat*}{3}
R_3 = c^2 \int_0^{\xi} ((\Id - \Pi_{q-1}^t) \Delta \ut, \wt)_{\Omega} \ds 
& \le c^2 \Cp(q - 1) \sum_{n = 1}^{m - 1} \tau_n \Norm{(\Id - \Pi_{q-1}^t) \Delta \ut}{L^1(\In; L^2(\Omega))} \Norm{\dpt \wt}{L^{\infty}(\In; L^2(\Omega))} \\
& \quad + c^2 \tau_m\Norm{(\Id - \Pi_{q-1}^t) \Delta \ut}{L^1(I_m; L^2(\Omega))} \Norm{\dpt \wt}{L^{\infty}(I_m; L^2(\Omega))}.
\end{alignat*}

Combining these bounds on~$\{R_i\}_{i = 1}^3$ with~\eqref{eq:final-identity-posteriori}, we can estimate~$\Norm{u - \utsup}{C^0([0, T]; L^2(\Omega))}$.

In order to estimate~$\Norm{\utsup - \ut}{C^0([0, T]; L^2(\Omega))}$,  we use~\eqref{eq:bound-Linf-uhtsup-uht} in Lemma~\ref{lemma:uhtsup-uht} (adapted to the semidiscrete case) to obtain
\begin{equation*}
\Norm{\utsup - \ut}{C^0([0, T]; L^2(\Omega))} \le \max_{n \in \{1, \ldots, N\}} \sqrt{C_{q, \star}\tau_n} \Norm{(\Id - \Pi_{q-1}^t) \vt}{L^2(\In; L^2(\Omega))}.
\end{equation*}
The result then follows from the triangle inequality.
\end{proof}
\section{Numerical experiments\label{sec:numerical-exp}}
In this section, we present some numerical experiments that validate our theoretical results. 
To do so, we use an object-oriented MATLAB implementation of the space--time FEM~\eqref{eq:space-time-formulation} (\texttt{Method I}), as well as the version using~\eqref{eq:space-time-equivalence-L2} (\texttt{Method II}). We recall that both formulations coincide in the case of homogeneous Dirichlet boundary conditions (see Remark~\ref{rem:equivalent-formulations}).
Several numerical experiments are available in literature (see~\cite{French_Peterson:1996,Zhao-Li:2016,Bause_etal:2020}); 
thus, we restrict ourselves to highlighting the aspects of the scheme related to this work.

We focus on the following error quantities:
\begin{equation}
\label{def:error-quantities}
\begin{split}
& \Norm{u - \uht}{C^0([0, T]; L^2(\Omega))}, \qquad \Norm{u - \uhtsup}{C^0([0, T]; L^2(\Omega))}, \\ & \Norm{v - \vht}{C^0([0, T]; L^2(\Omega))}, \qquad \Norm{\nabla (u - \uht)}{C^0([0, T]; L^2(\Omega)^d)},
\end{split}
\end{equation}
which are approximated taking the maximum value of the spatial errors on a set of uniformly distributed points on each time interval~$\In$.

\subsection{Nonhomogenous Dirichlet boundary conditions\label{sec:nonhomogeneous-dirichlet-experiment}}
Let the space--time domain~$\QT$ be given by~$(0, 1)^2 \times (0, 1)$. We consider the~$(2 + 1)$-dimensional problem~\eqref{eq:model-problem-hamiltonian} with initial conditions, Dirichlet boundary conditions, and source term chosen so that the exact solution is given by (cf.~\cite[Example~3.2]{Walkington:2014}) 
\begin{equation}
\label{eq:solution-nonhomogeneous-Dirichlet}
u(x, y, t) = \cos(\sqrt{2} \pi t) \cos(\pi x) \sin(\pi y), \qquad v(x, y, t) = -\sqrt{2} \pi \sin(\sqrt{2} \pi t) \cos(\pi x) \sin(\pi y).
\end{equation}

The corresponding Dirichlet datum~$\gD$ is nonzero when~$x = 0$ or~$x = 1$, and cannot be approximated exactly using Lagrange interpolation.

\paragraph{\texorpdfstring{$h$}{h}-convergence.} 
We first assess the~$h$-convergence of \texttt{Method I}. In Figure~\ref{fig:h-convergence}, we show (in~\emph{log-log} scale) the errors in~\eqref{def:error-quantities} corresponding to the problem with exact solution in~\eqref{eq:solution-nonhomogeneous-Dirichlet} for a set of structured simplicial meshes and approximations in space of degree~$p = 1, 2, 3$. In order to let the spatial error dominate, we have set a uniform time step~$\tau = 3.125 \times 10^{-2}$ and approximations in time of degree~$q = 4$.
In accordance with Theorem~\ref{thm:a-priori}, quasi-optimal convergence rates are observed in all cases, namely, of order~$\mathcal{O}(h^{p+1})$ for the~$C^0([0, T]; L^2(\Omega))$ errors, and of order~$\mathcal{O}(h^p)$ for the~$C^0([0, T]; H^1(\Omega))$ error.
\begin{figure}[!ht]
\centering
\includegraphics[width = 0.43\textwidth]{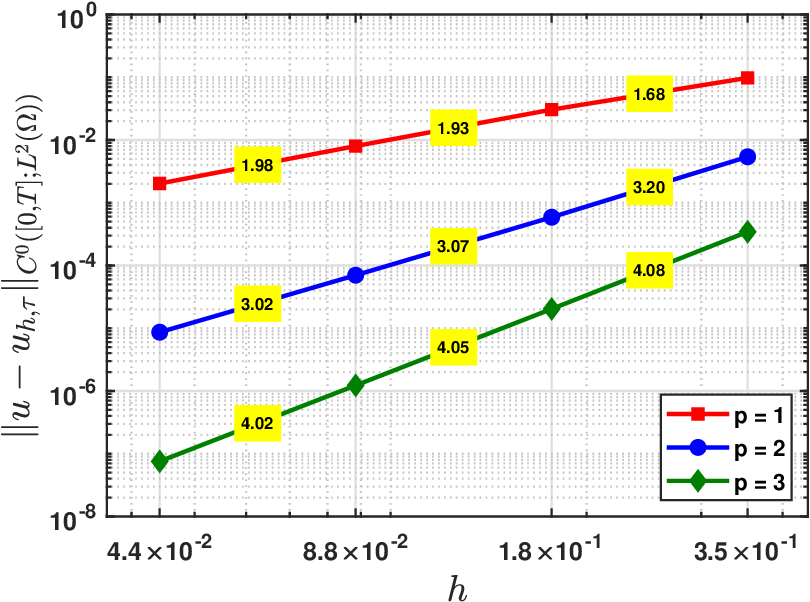}
\hspace{0.2in}
\includegraphics[width = 0.43\textwidth]{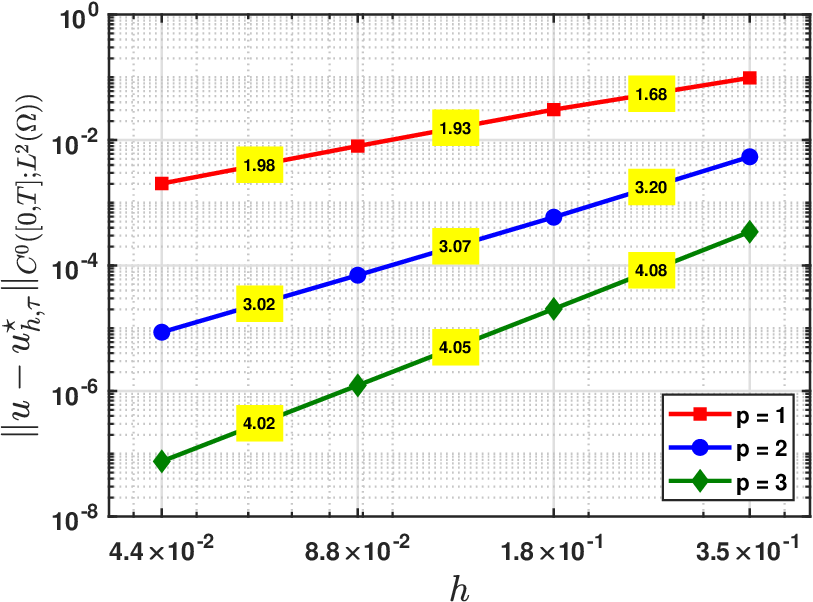}\\
\includegraphics[width = 0.43\textwidth]{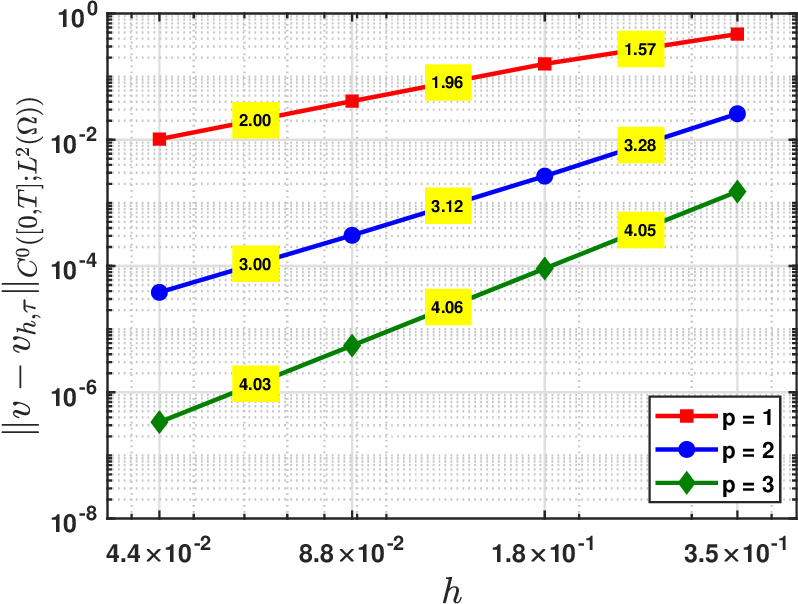}
\hspace{0.2in}
\includegraphics[width = 0.43\textwidth]{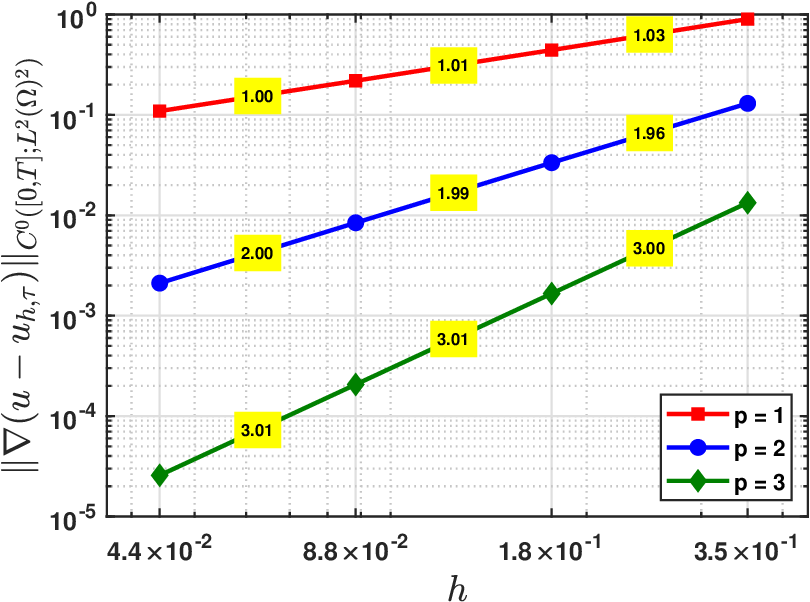}
\caption{$h$-convergence (in \emph{log-log} scale) of the errors in~\eqref{def:error-quantities} for~\texttt{Method I} corresponding to the problem with exact solution~$(u, v)$ in~\eqref{eq:solution-nonhomogeneous-Dirichlet}. The numbers in the yellow boxes are the empirical convergence rates. \label{fig:h-convergence}}
\end{figure}

\paragraph{Different formulations.} 
Before presenting the results for the~$\tau$-convergence of the method, which better illustrate the differences between~\texttt{Method I} and~\texttt{Method II}, 
we present a numerical experiment that stresses the fact that they both produce quantitatively different solutions for nonhomogeneous Dirichlet boundary conditions.

In Figure~\ref{fig:comparison-Grad-L2}, we plot the difference of the discrete solutions corresponding to~\texttt{Method I} and~\texttt{Method II} at time~$t_1 = 1/8$ using  a structured simplicial mesh with~$h  \approx 1.77 \times 10^{-1}$, a uniform time step~$\tau = 1.25 \times 10^{-1}$, and approximations of degrees~$p = q = 3$. 
Already after the first time step, the approximations on the elements close to the boundary are different, which subsequently propagates naturally to the whole domain.
\begin{figure}[!ht]
\centering
\includegraphics[width = 0.48\textwidth]{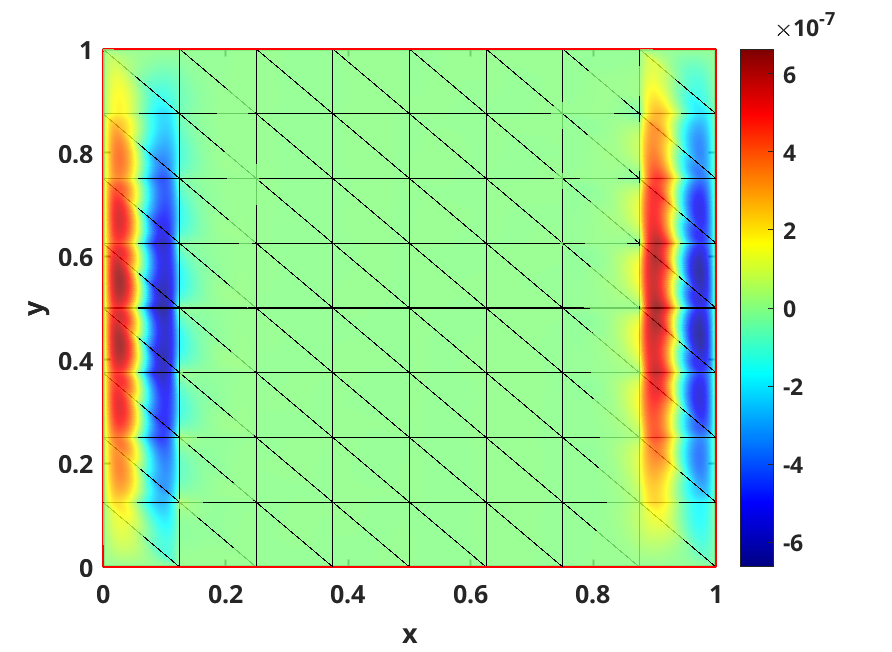}
\hspace{0.05in}
\includegraphics[width = 0.48\textwidth]{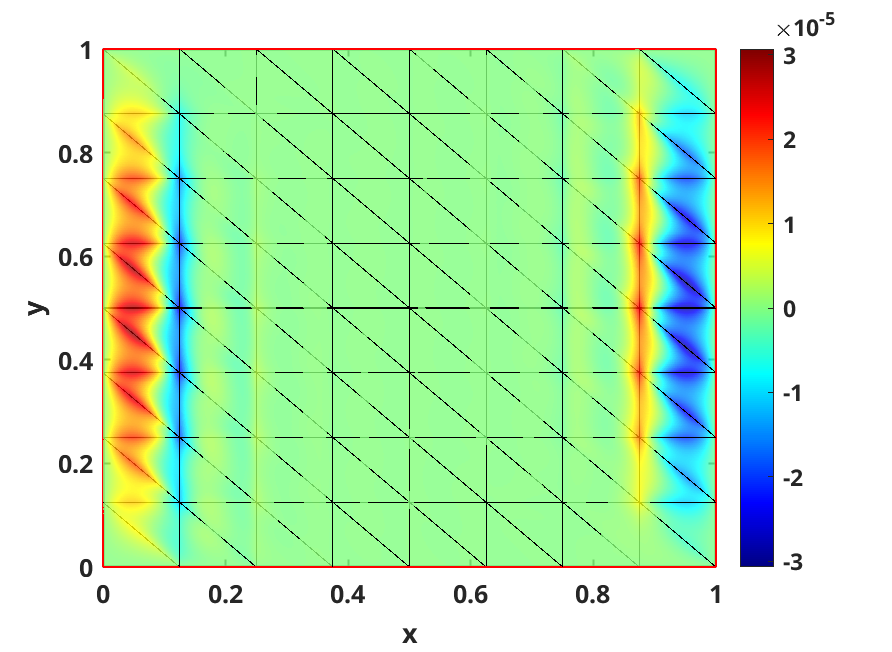}
\caption{
Plot of the difference of the discrete solutions~$\uht$ (left panel) and~$\vht$ (right panel) at~$t_1 = 1/8$ corresponding to~\texttt{Method I} and~\texttt{Method II} with~$p = q = 3$ for the problem with exact solution~$(u, v)$ in~\eqref{eq:solution-nonhomogeneous-Dirichlet}. 
\label{fig:comparison-Grad-L2}
}
\end{figure}

\paragraph{\texorpdfstring{$\tau$}{tau}-convergence.}
We now focus on the~$\tau$-convergence of~\texttt{Method I} and~\texttt{Method II}.
In Figure~\ref{fig:comparison-convergence-L2-Grad}, we show (in~\emph{log-log} scale) the errors in~\eqref{def:error-quantities} corresponding to the problem with exact solution in~\eqref{eq:solution-nonhomogeneous-Dirichlet} for a set of uniform time steps~$\tau = 2^{-(i + 1)}$ ($i = 1, 2, 3$) and approximations in time of degree~$q = 1, 2, 3, 4$.
Analogously as before, we have used a coarse simplicial mesh and approximations in space of degree~$p = 8$ so as to let the temporal error dominate.

When the Dirichlet boundary conditions are approximated according to Section~\ref{sec:discrete-data} (solid lines), both~\texttt{Method I} (left panel) and~\texttt{Method II} (right panel) show quasi-optimal convergence rates, namely, of order~$\mathcal{O}(\tau^{q+1})$ for~$\Norm{u - \uht}{C^0([0, T]; L^2(\Omega))}$, $\Norm{v - \vht}{C^0([0, T]; L^2(\Omega))}$, and~$\Norm{\nabla (u - \uht)}{C^0([0, T]; L^2(\Omega)^2)}$. 
Moreover, superconvergence of order~$\mathcal{O}(\tau^{q+2})$ is obtained for the postprocessed approximation~$\uhtsup$ when~$q > 1$ (see Remark~\ref{rem:superconvergence}).
On the other hand, \texttt{Method II} presents a degradation of the convergence rates when~$\gD$ is approximated using Lagrange interpolation (dashed lines).

\begin{figure}[!ht]
\centering
\includegraphics[width = 0.43\textwidth]{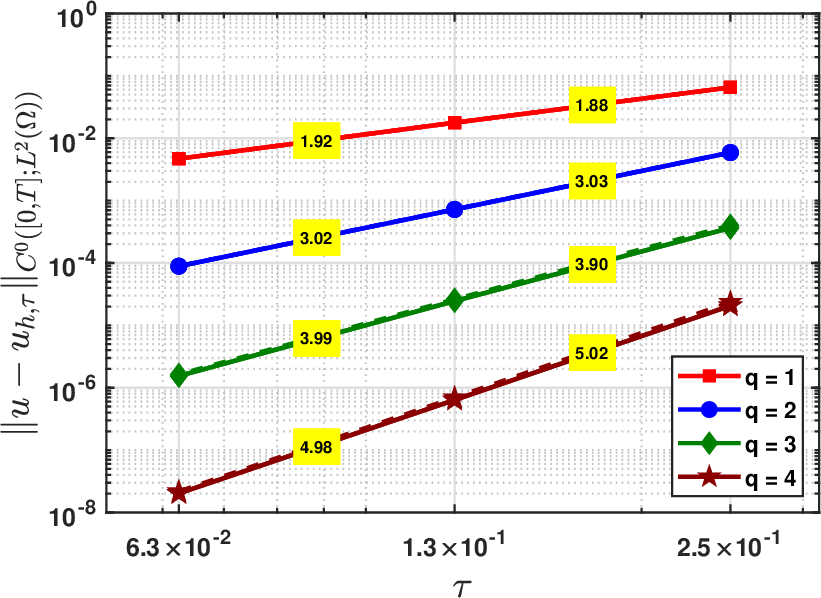}
\hspace{0.2in}
\includegraphics[width = 0.43\textwidth]{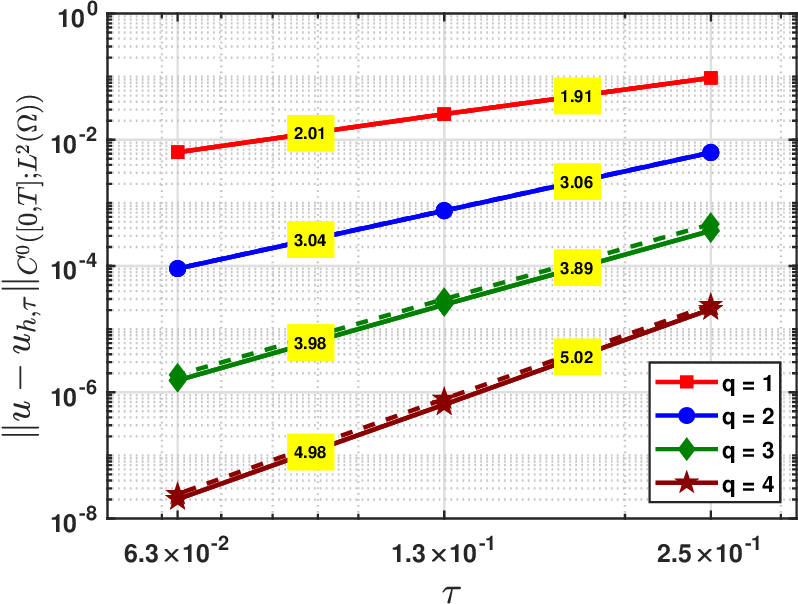} \\[1em]
\includegraphics[width = 0.43\textwidth]{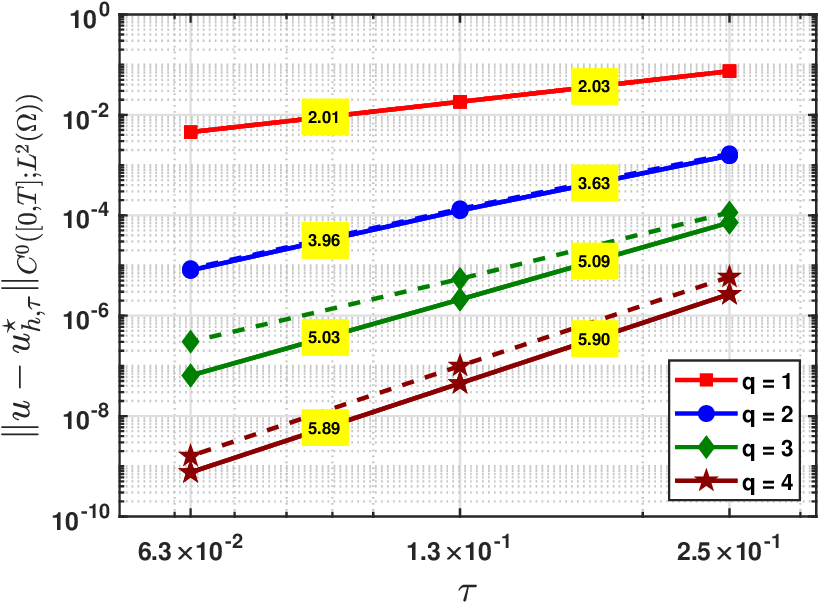}
\hspace{0.2in}
\includegraphics[width = 0.43\textwidth]{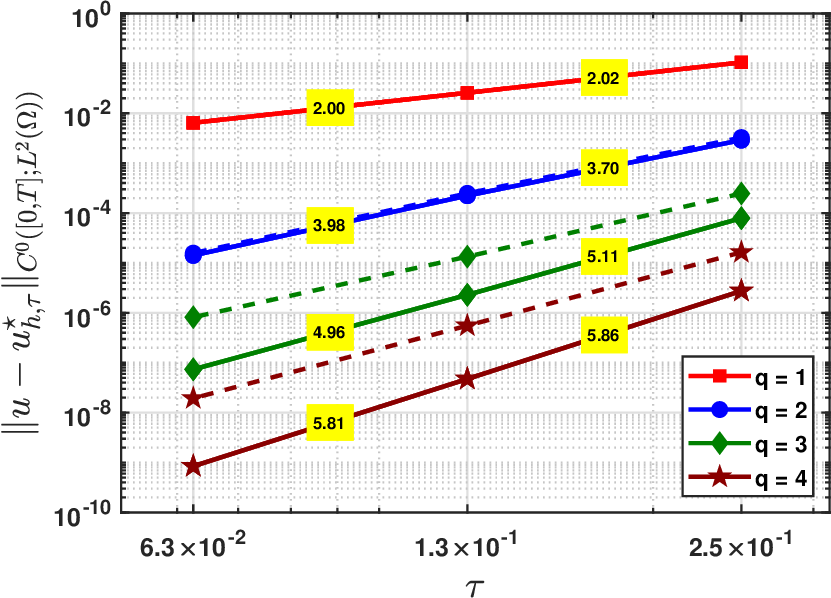} \\[1em]
\includegraphics[width = 0.43\textwidth]{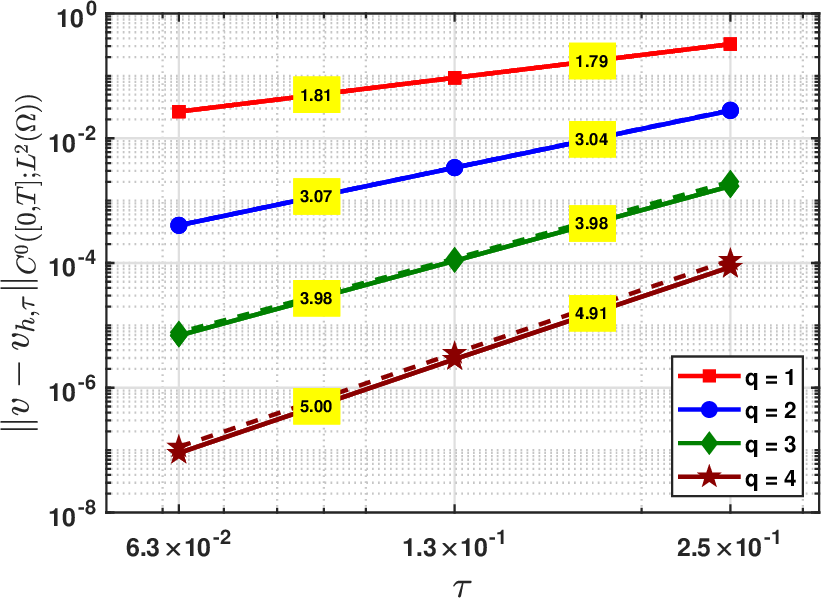}
\hspace{0.2in}
\includegraphics[width = 0.43\textwidth]{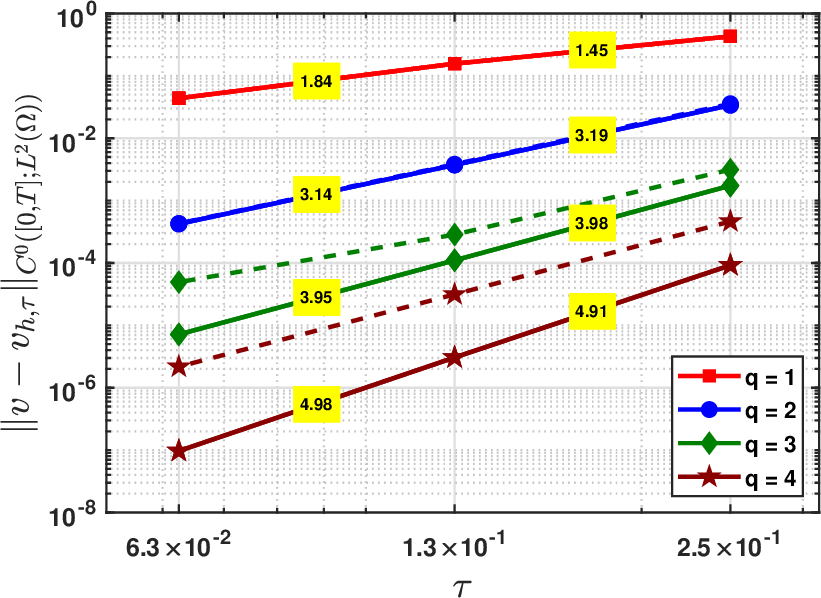} \\[1em]
\includegraphics[width = 0.43\textwidth]{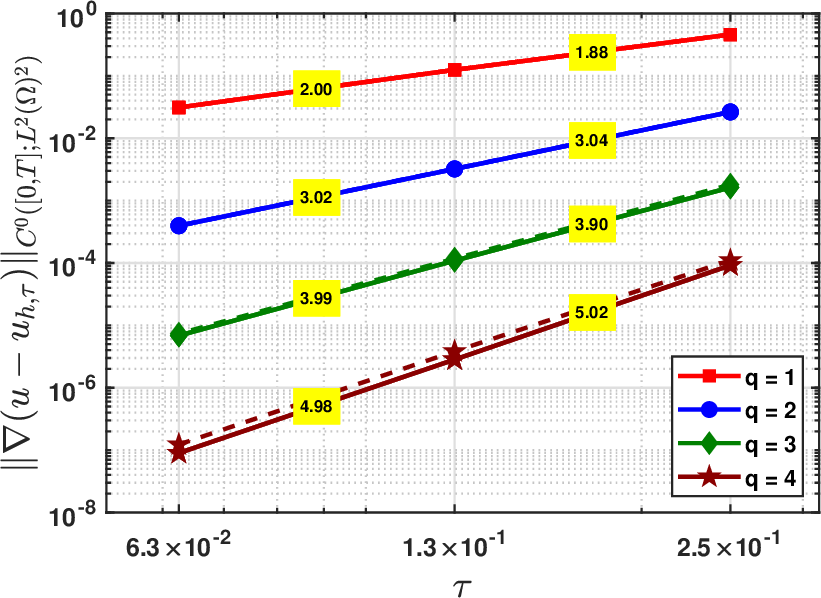}
\hspace{0.2in}
\includegraphics[width = 0.43\textwidth]{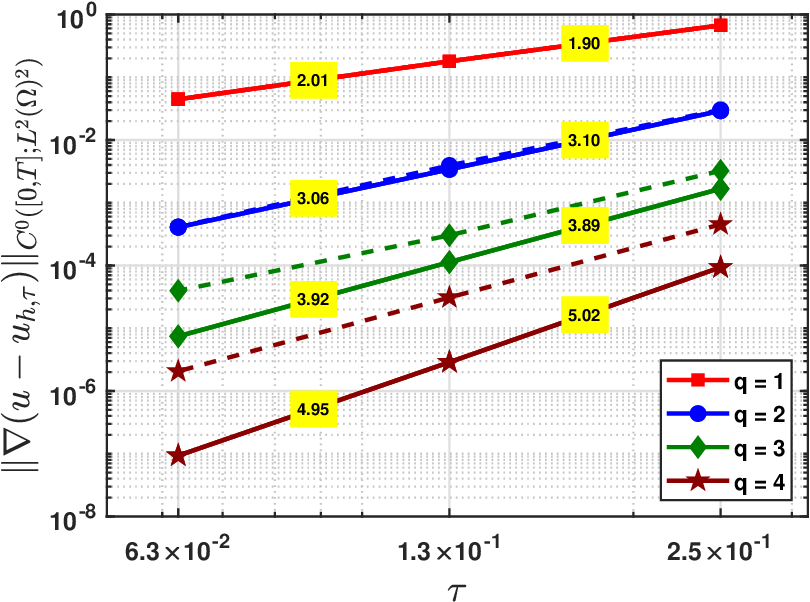}
\caption{$\tau$-convergence (in \emph{log-log} scale) of the errors in~\eqref{def:error-quantities} for~\texttt{Method I} (left panel) and~\texttt{Method II} (right panel) corresponding to the problem with exact solution~$(u, v)$ in~\eqref{eq:solution-nonhomogeneous-Dirichlet}. The results for the approximation of the Dirichlet boundary conditions according to Section~\ref{sec:discrete-data} are shown in solid lines, whereas those corresponding to Lagrange interpolation are shown in dashed lines. \label{fig:comparison-convergence-L2-Grad}}
\end{figure}

\paragraph{\texorpdfstring{$(p, q)$}{(p, q)}-convergence}
One of the main advantages of space--time methods is the possibility to obtain simultaneous high-order accuracy in space and time. In fact, when the solution is analytic, an exponential decay of the errors is expected if one fixes the space--time mesh and increase only the approximation degrees, resulting in the~$(p, q)$-version of the method. 

We fix a space--time mesh with~$h \approx 1.77 \times 10^{-1}$ and~$\tau = 0.25$. 
In Figure~\ref{fig:pq-convergence}, we present~(in \emph{semilogy} scale) the errors in~\eqref{def:error-quantities} for~\texttt{Method I} with approximations of degrees~$p = q$ and~$q = 1, \ldots, 6$. An exponential convergence of order~$\mathcal{O}(e^{- b \sqrt[3]{N_{\mathrm{DoFs}}}})$, where~$\mathrm{N}_{\mathrm{DoFs}}$ is the total number of degrees of freedom, is observed.

\begin{figure}[!ht]
\centering
\includegraphics[width = 0.43\textwidth]{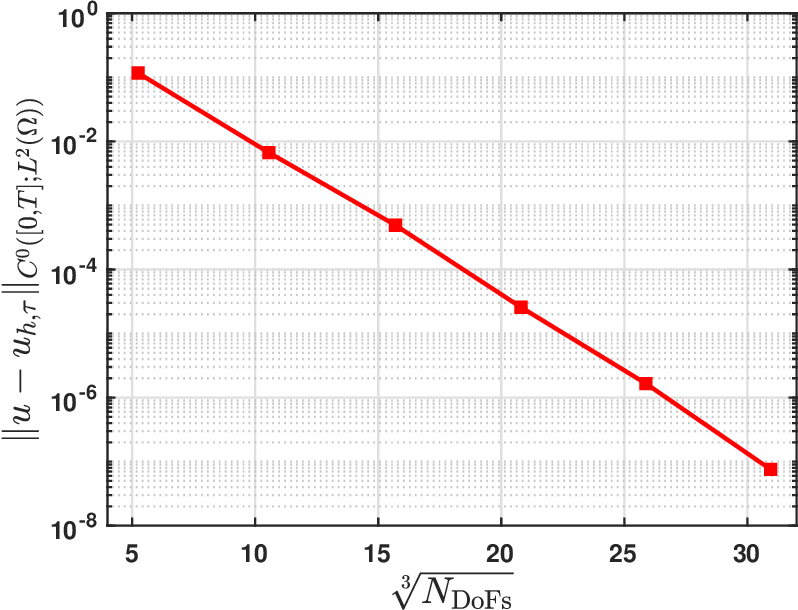}
\hspace{0.2in}
\includegraphics[width = 0.43\textwidth]{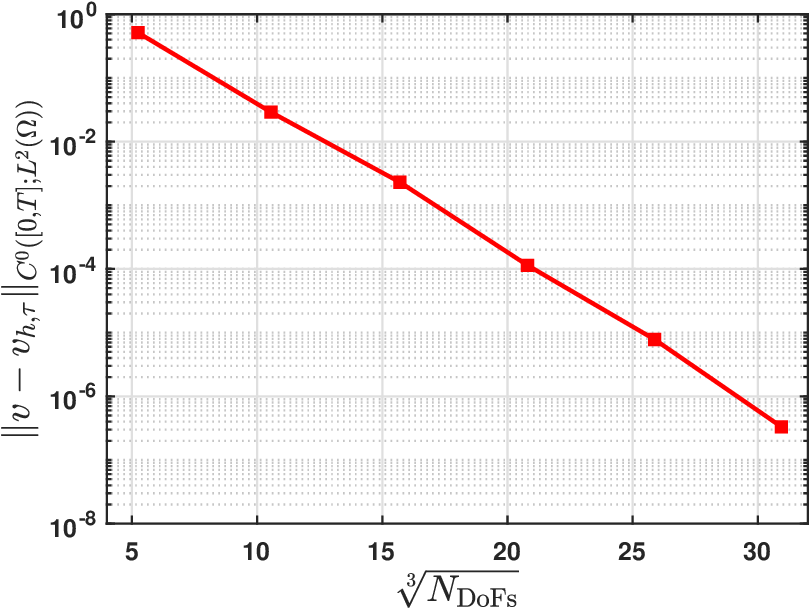}
\caption{$(p, q)$-convergence (in \emph{semilogy} scale) of the errors in~\eqref{def:error-quantities} for~\texttt{Method I} corresponding to the problem with exact solution~$(u, v)$ in~\eqref{eq:solution-nonhomogeneous-Dirichlet}. \label{fig:pq-convergence}}
\end{figure}

\subsection{\emph{A posteriori} error estimator}
We now focus on the~\emph{a posteriori} error estimate in Theorem~\ref{thm:a-posteriori}, which can be written as
\begin{equation*}
\Norm{u - \ut}{C^0([0, T]; L^2(\Omega))} \le \eta + \mathrm{osc}(f),
\end{equation*}
where~$\mathrm{osc}(f)$ is the term corresponding to the oscillation of the data.

We will validate the reliability and study numerically the efficiency of the error estimator~$\eta$. To do so, we define the effectivity index as
\begin{equation}
\label{eq:effectivity-index}
\text{Effectivity index} := \frac{\eta}{\Norm{\eu}{C^0([0, T]; L^2(\Omega))}}.
\end{equation}

Let~$\QT$ be given by~$(-1, 1)^2 \times (0, 1)$. We consider the~$(2 + 1)$-dimensional problem~\eqref{eq:model-problem-hamiltonian} with initial conditions, Dirichlet boundary conditions, and source term chosen so that the exact solution is given by (cf. \cite[\S4]{Dong-Mascotto-Wang:2024})
\begin{equation}
\label{eq:solution-posteriori}
u(x, y, t) = \psi(t) (1 - x^2) (1 - y^2) \quad v(x, y, t) = \psi'(t) (1 - x^2) (1 - y^2),
\end{equation}
for either~$\psi(t) = \cos(4t)$ or~$\psi(t) = t^{\alpha}$ for some~$\alpha > 0$. In the rest of this section, we set~$p = 4$ and a simplicial mesh with~$h = \sqrt{2}$ so that the spatial errors can be neglected.

\paragraph{Smooth solution.} We first consider the problem with exact solution~$(u, v)$ in~\eqref{eq:solution-posteriori} with~$\psi(t) = \cos(4t)$, so that~$u, v \in C^{\infty}(\QT)$.

In Figure~\ref{fig:estimator-smooth}(left panel), we compare the error~$\Norm{\eu}{C^0([0, T]; L^2(\Omega))}$ (solid lines) with the estimator~$\eta$ (dashed lines) for a sequence of time steps~$\tau$ and approximations in time of degree~$q = 1, 2, 3$. The corresponding effectivity indices are shown in Figure~\ref{fig:estimator-smooth}(right panel), which remain stable with respect to~$\tau$.

\begin{figure}[!ht]
\centering
\includegraphics[width = 0.43\textwidth]{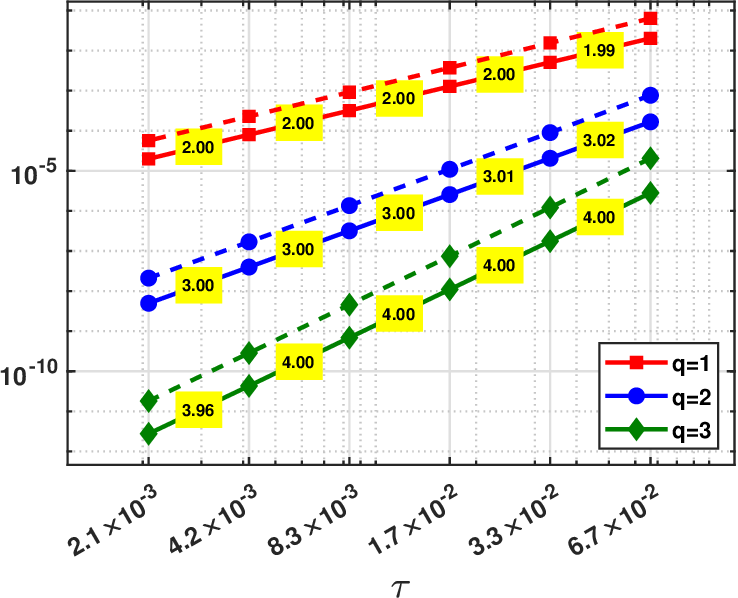}
\hspace{0.3in}
\includegraphics[width = 0.43\textwidth]{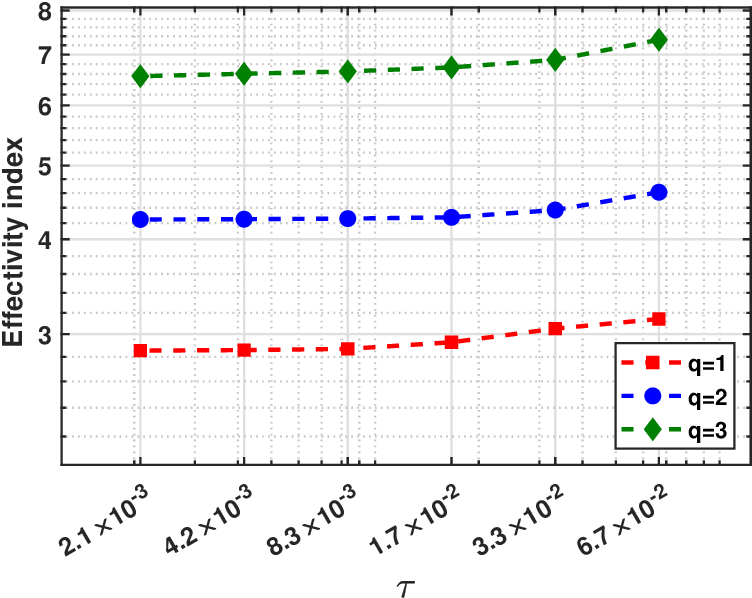}
\caption{\textbf{Left panel:} comparison of the error~$\Norm{\eu}{C^0([0, T]; L^2(\Omega))}$ (solid lines) and the estimator~$\eta$ (dashed lines) for the problem with smooth exact solution~$(u, v)$ in~\eqref{eq:solution-posteriori} and~$\psi(t) = \cos(4t)$. \textbf{Right panel:}  corresponding effectivity indices in~\eqref{eq:effectivity-index}. \label{fig:estimator-smooth}}
\end{figure}

\paragraph{Singular solution.} Similarly, we consider the problem with exact solution~$(u, v)$ in~\eqref{eq:solution-posteriori} and~$\psi(t) = t^{\alpha}$ for~$\alpha = 2.25$ and~$\alpha = 2.5$, so that~$u \in H^{\alpha + 1/2 - \varepsilon}(0, T; C^{\infty}(\Omega))$ for all~$\varepsilon > 0$.

In Figure~\ref{fig:estimator-singular}(left panel), we compare the error~$\Norm{\eu}{C^0([0, T]; L^2(\Omega))}$ (solid lines) with the estimator~$\eta$ (dashed lines) for~$\alpha = 2.25$ (first row) and~$\alpha = 2.5$ (second row) with quadratic approximations in time.
The same convergence rates of order~$\mathcal{O}(\tau^{\alpha})$ are observed in both cases, which are suboptimal by half an order.
The corresponding effectivity indices are shown in Figure~\ref{fig:estimator-singular}(right panel), which, as before, remain stable with respect to~$\tau$.
\begin{figure}[!ht]
\centering
\includegraphics[width = 0.43\textwidth]{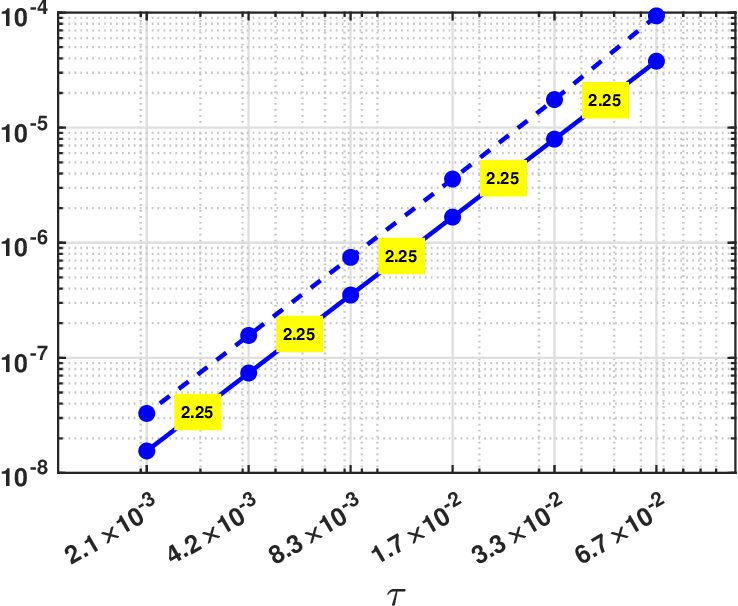}
\hspace{0.3in}
\includegraphics[width = 0.43\textwidth]{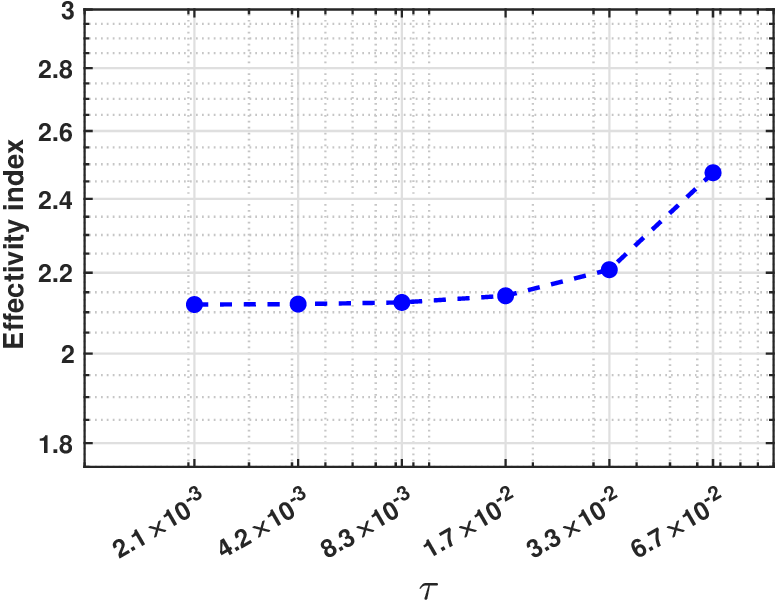} \\[1em]
\includegraphics[width = 0.43\textwidth]{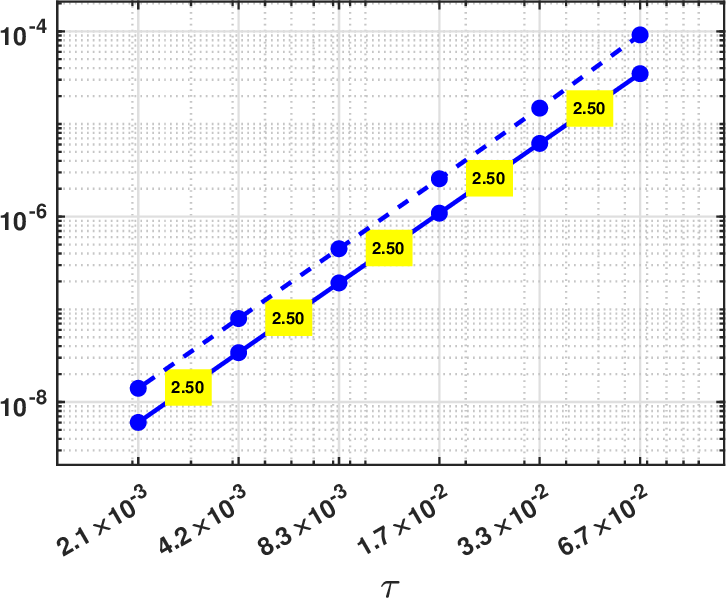}
\hspace{0.3in}
\includegraphics[width = 0.43\textwidth]{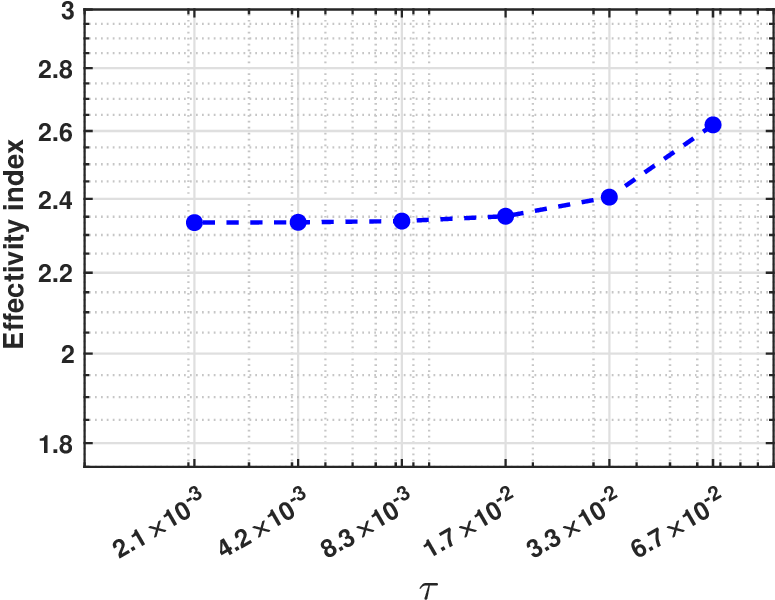}
\caption{\textbf{Left panel:} comparison of the error~$\Norm{\eu}{C^0([0, T]; L^2(\Omega))}$ (solid lines) and the estimator~$\eta$ (dashed lines) for the problem with singular exact solution~$(u, v)$ in~\eqref{eq:solution-posteriori} and~$\psi(t) = t^{2.25}$ (first row),  $\psi(t) = t^{2.5}$ (second row). \textbf{Right panel:}  corresponding effectivity indices in~\eqref{eq:effectivity-index}. \label{fig:estimator-singular}}
\end{figure}

\section{Concluding remarks\label{sec:conclusions}}
We have carried out a stability and convergence analysis of the continuous space--time FEM for the Hamiltonian formulation of the wave equation with possible nonhomogeneous Dirichlet boundary conditions.
By using some nonstandard test functions, we proved a continuous dependence of the discrete solution on the data of the problem in some~$C^0([0, T]; X)$-type norms, which is then used to derive 
\emph{a priori} error estimates with quasi-optimal convergence rates.
Based on the properties of a postprocessed approximation~$\uhtsup$, we also derived a reliable~\emph{a posteriori} estimate of the error in the~$C^0([0, T]; L^2(\Omega))$ norm.
Our numerical experiments show the expected convergence rates in all norms, as well as superconvergence in time of order~$\mathcal{O}(\tau^{q + 2})$ for the postprocessed approximation~$\uhtsup$ when~$q > 1$.
Efficiency and reliability of the error estimator have been numerically observed  for some test cases with smooth and singular solutions. 
We hope that the present analysis will be useful for studying more complex wave models. 

\section*{Acknowledgments}
The author acknowledges support from the
Italian Ministry of University and Research through the project PRIN2020 ``Advanced polyhedral discretizations of heterogeneous PDEs for multiphysics problems", and from the INdAM-GNCS through the
project CUP E53C23001670001.
The author is also grateful to Ilaria Perugia and Matteo Ferrari (University of Vienna) for kindly discussing this work.


\begin{thebibliography}{10}

\bibitem{Acosta_Duran:2004}
G.~Acosta and R.~G. Dur\'an.
\newblock An optimal {P}oincar\'e{} inequality in {$L^1$} for convex domains.
\newblock {\em Proc. Amer. Math. Soc.}, 132(1):195--202, 2004.

\bibitem{Adjerid:2002}
S.~Adjerid.
\newblock A posteriori finite element error estimation for second-order
  hyperbolic problems.
\newblock {\em Comput. Methods Appl. Mech. Engrg.}, 191(41-42):4699--4719,
  2002.

\bibitem{Anselmann_Bause:2020}
M.~Anselmann and M.~Bause.
\newblock Numerical study of {Galerkin}-collocation approximation in time for
  the wave equation.
\newblock In {\em Mathematics of wave phenomena. Selected papers based on the
  presentations at the conference, Karlsruhe, Germany, July 23--27, 2018},
  pages 15--36. Cham: Birkh{\"a}user, 2020.

\bibitem{Anselmann_etal:2020}
M.~Anselmann, M.~Bause, S.~Becher, and G.~Matthies.
\newblock Galerkin-collocation approximation in time for the wave equation and
  its post-processing.
\newblock {\em ESAIM Math. Model. Numer. Anal.}, 54(6):2099--2123, 2020.

\bibitem{Antonietti_etal:2018}
P.~F. Antonietti, I.~Mazzieri, N.~Dal~Santo, and A.~Quarteroni.
\newblock A high-order discontinuous {G}alerkin approximation to ordinary
  differential equations with applications to elastodynamics.
\newblock {\em IMA J. Numer. Anal.}, 38(4):1709--1734, 2018.

\bibitem{Antonietti_Mazzieri_Migliorini:2020}
P.~F. Antonietti, I.~Mazzieri, and F.~Migliorini.
\newblock A space-time discontinuous {G}alerkin method for the elastic wave
  equation.
\newblock {\em J. Comput. Phys.}, 419:109685, 26, 2020.

\bibitem{Aziz_Monk:1989}
A.~K. Aziz and P.~Monk.
\newblock Continuous finite elements in space and time for the heat equation.
\newblock {\em Math. Comp.}, 52(186):255--274, 1989.

\bibitem{Bales-Lasiecka:1994}
L.~Bales and I.~Lasiecka.
\newblock Continuous finite elements in space and time for the nonhomogeneous
  wave equation.
\newblock {\em Comput. Math. Appl.}, 27(3):91--102, 1994.

\bibitem{Bales-Lasiecka:1995}
L.~Bales and I.~Lasiecka.
\newblock Negative norm estimates for fully discrete finite element
  approximations to the wave equation with nonhomogeneous {$L_2$} {D}irichlet
  boundary data.
\newblock {\em Math. Comp.}, 64(209):89--115, 1995.

\bibitem{Bangerth_etal:2010}
W.~Bangerth, M.~Geiger, and R.~Rannacher.
\newblock Adaptive {G}alerkin finite element methods for the wave equation.
\newblock {\em Comput. Methods Appl. Math.}, 10(1):3--48, 2010.

\bibitem{Bause_etal:2020}
M.~Bause, U.~K\"ocher, F.~A. Radu, and F.~Schieweck.
\newblock Post-processed {G}alerkin approximation of improved order for wave
  equations.
\newblock {\em Math. Comp.}, 89(322):595--627, 2020.

\bibitem{Bernardi_Suli:2005}
C.~Bernardi and E.~S\"uli.
\newblock Time and space adaptivity for the second-order wave equation.
\newblock {\em Math. Models Methods Appl. Sci.}, 15(2):199--225, 2005.

\bibitem{Brenner-Scott:book}
S.~C. Brenner and L.~R. Scott.
\newblock {\em The mathematical theory of finite element methods}, volume~15 of
  {\em Texts in Applied Mathematics}.
\newblock Springer, New York, third edition, 2008.

\bibitem{Brezis:2011}
H.~Brezis.
\newblock {\em Functional analysis, {S}obolev spaces and partial differential
  equations}.
\newblock Universitext. Springer, New York, 2011.

\bibitem{Cangiani_Dong_Georgoulis:2017}
A.~Cangiani, Z.~Dong, and E.~H. Georgoulis.
\newblock {$hp$}-version space-time discontinuous {G}alerkin methods for
  parabolic problems on prismatic meshes.
\newblock {\em SIAM J. Sci. Comput.}, 39(4):A1251--A1279, 2017.

\bibitem{Chaumont_Ern:2024}
T.~Chaumont-Frelet and A.~Ern.
\newblock Damped energy-norm a posteriori error estimates for fully discrete
  approximations of the wave equation using {$C^2$}-reconstructions for the
  fully discrete wave equation with the leapfrog scheme.
\newblock Accepted for publication in ESAIM Math. Model. Numer. Anal.
  \href{https://doi.org/10.1051/m2an/2025027}{doi:10.1051/m2an/2025027}, 2025.

\bibitem{Dong-Mascotto-Wang:2024}
Z.~Dong, L.~Mascotto, and Z.~Wang.
\newblock A priori and a posteriori error estimates of a {DG}-{CG} method for
  the wave equation in second order formulation.
\newblock \href{http://arxiv.org/abs/2411.03264}{arXiv:2411.03264}, 2024.

\bibitem{Ern_Guermond-book-I}
A.~Ern and J.-L. Guermond.
\newblock {\em Finite elements {I}---{A}pproximation and interpolation},
  volume~72 of {\em Texts in Applied Mathematics}.
\newblock Springer, Cham, 2021.

\bibitem{Ern_Guermond-book-II}
A.~Ern and J.-L. Guermond.
\newblock {\em Finite elements {II}---{G}alerkin approximation, elliptic and
  mixed {PDE}s}, volume~73 of {\em Texts in Applied Mathematics}.
\newblock Springer, Cham, 2021.

\bibitem{Ern_Guermond-book-III}
A.~Ern and J.-L. Guermond.
\newblock {\em Finite elements {III}---first-order and time-dependent {PDE}s},
  volume~74 of {\em Texts in Applied Mathematics}.
\newblock Springer, Cham, 2021.

\bibitem{Evans:2010}
L.~C. Evans.
\newblock {\em Partial differential equations}, volume~19 of {\em Graduate
  Studies in Mathematics}.
\newblock American Mathematical Society, Providence, RI, second edition, 2010.

\bibitem{ferrari2024stability}
M.~Ferrari and S.~Fraschini.
\newblock Stability of conforming space-time isogeometric methods for the wave
  equation.
\newblock {\em Math. Comp.}, 2025.
\newblock \href{ https://doi.org/10.1090/mcom/4062}{doi:10.1090/mcom/4062}.

\bibitem{Ferrari_etal:2024}
M.~Ferrari, S.~Fraschini, G.~Loli, and I.~Perugia.
\newblock Unconditionally stable space-time isogeometric discretization for the
  wave equation in {H}amiltonian formulation.
\newblock Accepted for publication in ESAIM Math. Model. Numer. Anal., 2025.

\bibitem{Fraschini_etal:2024}
S.~Fraschini, G.~Loli, A.~Moiola, and G.~Sangalli.
\newblock An unconditionally stable space-time isogeometric method for the
  acoustic wave equation.
\newblock {\em Comput. Math. Appl.}, 169:205--222, 2024.

\bibitem{French:1993}
D.~A. French.
\newblock A space-time finite element method for the wave equation.
\newblock {\em Comput. Methods Appl. Mech. Engrg.}, 107(1-2):145--157, 1993.

\bibitem{French_Peterson:1991}
D.~A. French and T.~E. Peterson.
\newblock {\em A space-time finite element method for the second order wave
  equation}.
\newblock Centre for Mathematics and its Applications, Australian National
  University, 1991.

\bibitem{French_Peterson:1996}
D.~A. French and T.~E. Peterson.
\newblock A continuous space-time finite element method for the wave equation.
\newblock {\em Math. Comp.}, 65(214):491--506, 1996.

\bibitem{Georgoulis_etal:2016}
E.~H. Georgoulis, O.~Lakkis, C.~G. Makridakis, and J.~M. Virtanen.
\newblock A posteriori error estimates for leap-frog and cosine methods for
  second order evolution problems.
\newblock {\em SIAM J. Numer. Anal.}, 54(1):120--136, 2016.

\bibitem{Gomez-Nikolic:2024}
S.~G\'omez and V.~Nikoli\'c.
\newblock Space--time {DG}--{CG} finite element method for the {W}estervelt
  equation.
\newblock Accepted for publication in IMA J. Numer. Anal., 2025.

\bibitem{Grote_etal:2024}
M.~J. Grote, O.~Lakkis, and C.~Santos.
\newblock A posteriori error estimates for the wave equation with mesh change
  in the leapfrog method.
\newblock \href{https://doi.org/10.48550/arXiv.2411.16933}{arXiv:2411.16933},
  2024.

\bibitem{Hulbert_Hughes:1990}
G.~M. Hulbert and T.~J.~R. Hughes.
\newblock Space-time finite element methods for second-order hyperbolic
  equations.
\newblock {\em Comput. Methods Appl. Mech. Engrg.}, 84(3):327--348, 1990.

\bibitem{Johnson:1994}
C.~Johnson.
\newblock Discontinuous {G}alerkin finite element methods for second order
  hyperbolic problems.
\newblock {\em Comput. Methods Appl. Mech. Engrg.}, 107(1-2):117--129, 1993.

\bibitem{Karakashian_Makridakis:2005}
O.~Karakashian and C.~Makridakis.
\newblock Convergence of a continuous {G}alerkin method with mesh modification
  for nonlinear wave equations.
\newblock {\em Math. Comp.}, 74(249):85--102, 2005.

\bibitem{Lasiecka_Triggiani:1994}
I.~Lasiecka and R.~Triggiani.
\newblock Recent advances in regularity of second-order hyperbolic mixed
  problems, and applications.
\newblock In {\em Dynamics reported. Expositions in dynamical systems. New
  series. Volume 3}, pages 104--162. Berlin: Springer-Verlag, 1994.

\bibitem{Makridakis_Nochetto:2006}
C.~Makridakis and R.~H. Nochetto.
\newblock A posteriori error analysis for higher order dissipative methods for
  evolution problems.
\newblock {\em Numer. Math.}, 104(4):489--514, 2006.

\bibitem{Payne_Weinberger:1960}
L.~E. Payne and H.~F. Weinberger.
\newblock An optimal {P}oincar\'e{} inequality for convex domains.
\newblock {\em Arch. Rational Mech. Anal.}, 5:286--292, 1960.

\bibitem{Schotzau_Schwab:2000}
D.~Sch\"otzau and C.~Schwab.
\newblock Time discretization of parabolic problems by the {$hp$}-version of
  the discontinuous {G}alerkin finite element method.
\newblock {\em SIAM J. Numer. Anal.}, 38(3):837--875, 2000.

\bibitem{Schwab-book:1998}
Ch. Schwab.
\newblock {\em {$p$}- and {$hp$}-finite element methods}.
\newblock Numerical Mathematics and Scientific Computation. The Clarendon
  Press, Oxford University Press, New York, 1998.
\newblock Theory and applications in solid and fluid mechanics.

\bibitem{Scott-Zhang:1990}
L.~R. Scott and S.~Zhang.
\newblock Finite element interpolation of nonsmooth functions satisfying
  boundary conditions.
\newblock {\em Math. Comp.}, 54(190):483--493, 1990.

\bibitem{Strauss_Vazquez:1978}
W.~Strauss and L.~Vazquez.
\newblock Numerical solution of a nonlinear {K}lein-{G}ordon equation.
\newblock {\em J. Comput. Phys.}, 28(2):271--278, 1978.

\bibitem{Tantardini_Veeser:2016}
F.~Tantardini and A.~Veeser.
\newblock The {$L^2$}-projection and quasi-optimality of {G}alerkin methods for
  parabolic equations.
\newblock {\em SIAM J. Numer. Anal.}, 54(1):317--340, 2016.

\bibitem{Nagy:1941}
B.~v.~Sz.~Nagy.
\newblock {\"U}ber {I}ntegralungleichungen zwischen einer {F}unktion und ihrer
  {A}bleitung.
\newblock {\em Acta Univ. Szeged. Sect. Sci. Math.}, 10:64--74, 1941.

\bibitem{Walkington:2014}
N.~J. Walkington.
\newblock Combined {DG}-{CG} time stepping for wave equations.
\newblock {\em SIAM J. Numer. Anal.}, 52(3):1398--1417, 2014.

\bibitem{Zank:2019}
M.~Zank.
\newblock Inf--sup stable space--time methods for time--dependent partial
  differential equations. volume 36 of {M}onographic {S}eries {TU} {G}raz:
  {C}omputation in {E}ngineering and {S}cience.
\newblock {\em Verlag der TU Graz, Graz}, 2019.

\bibitem{Zank:2019proc}
M.~Zank.
\newblock The {N}ewmark method and a space-time {FEM} for the second-order wave
  equation.
\newblock In {\em Numerical mathematics and advanced applications---{ENUMATH}
  2019}, volume 139 of {\em Lect. Notes Comput. Sci. Eng.}, pages 1225--1233.
  Springer, Cham, 2021.

\bibitem{Zhao-Li:2016}
Z.~Zhao and H.~Li.
\newblock Convergence of a space-time continuous {G}alerkin method for the wave
  equation.
\newblock {\em J. Inequal. Appl.}, pages Paper No. 280, 18, 2016.

\end{thebibliography}
\end{document}